\documentclass[a4paper,12pt,oneside]{article}

\usepackage{mathtools}
\usepackage{tikz}
\usetikzlibrary{shapes,arrows,positioning} 
\tikzset{
    >=stealth',
    punkt/.style={
           rectangle,
           rounded corners,
           draw=black, very thick, 
           text width=8em,
           minimum height=2em,
           text centered, fill=white, drop shadow},
    punkta/.style={
           rectangle,
           rounded corners,
           draw=black, very thick,
           text width=10em,
           minimum height=2em,
           text centered, fill=white, drop shadow},
    punktaka/.style={
           rectangle,
           rounded corners,
           draw=black, very thick,
           text width=14em,
           minimum height=2em,
           text centered, fill=white, drop shadow},       
    punktaa/.style={
           rectangle,
           rounded corners,
           draw=black, very thick,
           text width=15em,
           minimum height=2em,
           text centered, fill=white, drop shadow},
    punktaaa/.style={
           rectangle,
           rounded corners,
           draw=black, very thick,
           text width=10em,
           minimum height=2em,
           text centered, fill=white, drop shadow},
    pil/.style={
           ->,
           thick,
           shorten <=2pt,
           shorten >=2pt,}
}

\usepackage[framemethod=tikz]{mdframed}
\usepackage{soul}
\usepackage{transparent}
\usetikzlibrary{arrows, decorations.markings}
\usetikzlibrary{automata,positioning}

\usepackage{tikz}
\usetikzlibrary{shadings,shadows,shapes.arrows}

\usetikzlibrary{shapes,snakes}
\usetikzlibrary{matrix}

\usepackage{fancybox}

\usepackage{etex}
\usepackage{tcolorbox}
\usepackage{color}
\usepackage{fancybox,framed}

\usepackage[framemethod=tikz]{mdframed}
\usepackage{lipsum}
\usepackage{amssymb}
\usepackage{pifont}
\definecolor{mycolor}{rgb}{0.122, 0.435, 0.698}
\usepackage{emptypage}

\usepackage{afterpage}

\usepackage{epigraph}
\usepackage{url}
\usepackage{fancyhdr}

\usepackage{enumitem}
\usepackage{wrapfig}
\usepackage[toc,page]{appendix}

\usepackage{tikz}
\usepackage{fp}
\usepackage{pgf}
\usepackage{xifthen}
\newcommand{\hawaiiApp}[2] 
{\FPeval{\points}{4-((#1)/2)}
	\begin{tikzpicture}[scale=5,domain=0:1] 
	\tikzstyle{every node}=[circle, draw, fill=black!50,
	inner sep=0pt, minimum width=\points pt]
	\FPeval{\step}{1/2^((#1))} 
	\FPeval{\stepone}{1/2^(((#1))-1)}
	\FPeval{\kminustwo}{((#1))-1}
	\ifthenelse{#2=1}{
		\draw[step=\stepone,gray,very thin] (-0.1,-0.1) grid (1.1,1.1);}{}	
	\draw (0,0)--(1,0)--(1,1)--(0,1)--cycle;
	\draw (0,0.5)--(0.5,0.5)--(0.5,0);
	\draw (0,0.25)--(0.25,0.25)--(0.25,0);
	\draw (0,0.125)--(0.125,0.125)--(0.125,0);
	\draw (0,1/16)--(1/16,1/16)--(1/16,0);
	\draw (0,1/32)--(1/32,1/32)--(1/32,0);
	\draw (0,1/64)--(1/64,1/64)--(1/64,0);
	\foreach \x in {0,\stepone,...,1} {\draw node at (\x,0) {};
		\draw node at (\x,1) {};
		\draw node at (0,\x) {};
		\draw node at (1,\x) {};
	}
	\foreach \x in {0,...,\kminustwo} {
		\pgfmathsetmacro{\m}{2^(\x)*\stepone}													
		\foreach \z in {0,\stepone,...,\m} {
			\draw node at (\z,\m) {} ;
			\draw node at (\m,\z) {} ;
		}
	}
	\ifthenelse{#2=0}{
		\draw node at (\step,\step){};
	}{}
	\end{tikzpicture}}

\newcommand{\hawaiiPol}[1] 
{\FPeval{\points}{4-((#1)/2)}
	\begin{tikzpicture}[scale=5,domain=0:1] 
	\tikzstyle{every node}=[circle, draw, fill=black!50,
	inner sep=0pt, minimum width=\points pt]
	\FPeval{\step}{1/2^((#1))} 
	\FPeval{\stepone}{1/2^(((#1))-1)}
	\FPeval{\kminustwo}{((#1))-1}
	\draw (0,0)--(1,0)--(1,1)--(0,1)--cycle;
	\draw (\stepone,\stepone)--(\stepone,2*\stepone);
	\draw (\stepone,\stepone)--(2*\stepone,\stepone);
	\foreach \x in {0,...,\kminustwo} {
		\pgfmathsetmacro{\m}{2^(\x)*\stepone}	
		\draw (0,\m)--(\m,\m);	
		\draw (\m,0)--(\m,\m);																							
	}
	\draw  [fill=gray, fill opacity=0.85] (\step,\step)--(0,0)--(0,\stepone)--cycle;
	\draw  [fill=gray, fill opacity=0.85] (\step,\step)--(0,0)--(\stepone,0)--cycle;
	\draw  [fill=gray, fill opacity=0.85] (\step,\step)--(0,\stepone)--(\stepone,\stepone)--cycle;
	\draw  [fill=gray, fill opacity=0.85] (\step,\step)--(\stepone,0)--(\stepone,\stepone)--cycle;
	\foreach \x in {0,\stepone,...,1} {\draw node at (\x,0) {};
		\draw node at (\x,1) {};
		\draw node at (0,\x) {};
		\draw node at (1,\x) {};
	}
	\foreach \x in {0,...,\kminustwo} {
		\pgfmathsetmacro{\m}{2^(\x)*\stepone}													
		\foreach \z in {0,\stepone,...,\m} {
			\draw node at (\z,\m) {} ;
			\draw node at (\m,\z) {} ;
		}
	}
	\draw node at (\step,\step){};
	\end{tikzpicture}}

\usepackage[utf8]{inputenc}
\usepackage{color}
\usepackage{import}
\usepackage{setspace}
\usepackage{latexsym,amsmath,amsfonts,amscd,amssymb,amsthm}
\usepackage{graphicx}
\usepackage[all]{xy}
\usepackage{epstopdf}
\DeclareGraphicsExtensions{.pdf,.png,.jpg,.gif, .eps}

\usepackage{amsthm}
\usepackage{thmtools}

\theoremstyle{plain} 
\newtheorem{teo}{Theorem}

\newtheorem{prop}{Proposition}
\newtheorem{cor}{Corollary}
\theoremstyle{definition}

\newtheorem{preg}{Question}

\theoremstyle{remark}
\newtheorem{obs}{Remark}
\newtheorem{ej}{Example}

\newcommand{\parentesis}[1]{\left(#1\right)}
\newcommand{\comillas}[1]{``#1''}
\newcommand{\punto}[2]{\left(#1,#2\right)}
\newcommand{\conjunto}[1]{\left\lbrace #1 \right\rbrace}
\newcommand{\mapeo}[5]{
\begin{eqnarray*}
#1:#2 & \longrightarrow & #3\\
#4 & \longmapsto & #5
\end{eqnarray*}}
\newcommand{\cms}{$(X,\textrm{d})$ }

\newcommand{\todon}{n\in\mathbb{N}}
\renewcommand{\leq}{\leqslant}
\renewcommand{\epsilon}{\varepsilon}
\newcommand{\ball}[2]{\textrm{B}(#1,#2)}
\newcommand{\dist}[2]{\textrm{d}(#1,#2)}

\newcommand{\diam}{\textrm{diam}}

\usepackage{hyperref}
\hypersetup{
	colorlinks=true,
	linkcolor=blue,
	filecolor=magenta,      
	urlcolor=cyan,
	pdfpagemode=FullScreen,
}

\usepackage[math]{iwona}
\usepackage[T1]{fontenc}
\makeindex
\makeatletter
\newcommand{\subjclass}[2][2010]{%
  \let\@oldtitle\@title%
  \gdef\@title{\@oldtitle\footnotetext{#1 \emph{Mathematics Subject Classification.} #2}}%
}
\newcommand{\keywords}[1]{%
  \let\@@oldtitle\@title%
  \gdef\@title{\@@oldtitle\footnotetext{\emph{Key words and phrases.} #1.}}%
}
\makeatother
\title{Polyhedral expansions of compacta associated to finite approximations}
\usepackage{authblk}
\author{Diego Mondéjar}
\affil{\small{Departamento de Matemática Aplicada y Estadística\protect\\Universidad San Pablo CEU, Madrid, Spain\protect\\ {\fontfamily{pcr}\selectfont
		diego.mondejarruiz@ceu.es}}}
\date{}                     

\subjclass{54B20, 54C56, 54C60, 54D10, 55P55, 55Q07, 55U05}
\keywords{Finite topological spaces, Alexandroff spaces, Hyperspaces, Shape Theory, Topological Persistence}
\begin{document}
\maketitle
\begin{abstract}
This paper introduces some inverse sequences of different polyhedra all based on finite approximations of a compact metric space so they can be used to capture the shape type of the original space. It is shown that they are HPol-expansions, proving the so-called General Principle. We use these sequences to compute explicitly some inverse persistent homology groups of a space and measure its errors in the approximation process.
\end{abstract}

\section{Introduction}\label{sec:intro}
The rise and development of Computational Topology \cite{EHcomputational, Ztopology}, Applied Algebraic Topology \cite{Gelementary} and Topological Data Analysis \cite{Ctopology} has boosted the study of topology and related fields. The idea of approximating topological spaces with simpler ones is of current interest, due to the possibility of actually compute topological properties of datasets. 

One example of this are the finite topological spaces. They are a special case of the more general Alexandroff spaces, in which the arbitrary intersections of open sets are open. The notion was introduced by Alexandroff in \cite{Adiskrete} but, due to its poor topological structure, their study got a bit stucked. The papers \cite{Sfinite,Msingular} were essential in their rediscovering, since it is shown that Alexandroff $T_0$ spaces can reproduce the same topology as polyhedra. An excellent reference about them are the notes of May \cite{Mfinitetopological,Mfinitecomplexes}. In the last years, due to its computational interest, there is plenty of theory devoted to them. See, for instance, \cite{BMsimple,BMone-point,BMautomorphism} or Barmak's book \cite{Balgebraic}.

Topological persistence \cite{Ctopology, Gbarcodes} uses finite sequences of polyhedra extracted from a dataset to reconstruct topological properties of the underlying space with some guarantees \cite{EHpersistent, ELZtopological, Estability, Vsketches}. These techniques have shown to be very effective and robust to samples with some noise. Moreover, they have been applied successfully to many different branches of science, such as Biology \cite{MunchBiology}, Financial Series \cite{GideaFinancial} or Natural Language Processing \cite{ElyasiNLP}.

An old procedure of approximating topological spaces is by means of inverse sequences of simpler spaces. From the works of Alexandroff \cite{Auntur} and Freudenthal \cite{Freudenthal}, to more modern results \cite{KTWthe, KWfinite, Cinverse, Bilskiinverse, MMreconstruction}, there are a lot of results for concrete types of spaces, both approximate and approximators. Shape Theory essentially uses this idea to enlarge the concept of homotopy of spaces to the ones with bad local properties, where homotopy does not detect any structure. For compact metric spaces, Borsuk \cite{Bconcerning} initiated this theory, using a sequence of neighborhoods of the space embedded in the Hilbert cube (see \cite{Btheory,BtheoryA}) and compare them, instead of the original spaces. After Borsuk, Shape Theory was extended to more general spaces (as Fox \cite{Fon}) and a whole theory was developed \cite{MSshape,DSshape,Mthirty,Mabsolute}. Vanessa Robins was the first to connect Topological Persistence and Shape Theory, using the latter to compute persistent homology groups of some attractors in \cite{Rtowards}. Gathering this idea, in \cite{MCLepsilon}, the authors defined an inductive construction over a compact metric space by terms of finite approximations made from it, called the \emph{Main Construction}. Giving these approximations an hyperspace topology, they conjectured that the inverse sequence of these finite spaces recover some topological properties of the original space and that the Alexandroff-McCord associated inverse sequence of polyhedra recover exactly its shape type, calling the later the \emph{General Principle}. In \cite{MMreconstruction}, it is showed that the inverse limit of the inverse sequence of finite spaces has the homotopy type of the original space. We show here that the inverse sequence of polyhedra (and some others defined in similar terms) actually define the shape of the space, answering positively to the conjecture of the General Principle. This paper is devoted to further connect Shape Theory and Topological Persistence, showing how to use this inverse sequences to compute some properties related to Inverse Persistence, introduced in \cite{MMreconstruction}.

The next section is used to recall the main definitions and tools used in the relationships established in the rest of the paper. In Section \ref{sec:GP}, we introduce some inverse sequences based on the Main Construction over a compact metric space, and called polyhedral approximative sequences, and show that they define the shape of the original space, giving a positive answer to the General Principle in Theorem \ref{teo:GP}. Hence, they are suitable for computing shape invariants, inverse persistence homology groups and inverse persistent homology errors, defined previously in Section \ref{sec:persistenterrors}. We finish giving a computable example in Section \ref{sec:computationalhawaiianearring}, where we explicitly construct a polyhedral approximative sequence of the Hawaiian Earring and compute all the elements defined in the previous section.
\section{Background and tools}\label{sec:tools}
Only essential and used results are mentioned, referring the reader to appropiate references to dig in details.
\paragraph{\textsc{Polyhedra}}
We recall the basic definitons and facts about them here and recommend \cite{Salgebraic,Mfinitecomplexes} and Appendix 1 of \cite{MSshape} for proofs and more information. An \emph{abstract simplicial complex} $K$ is a set of \emph{vertices} $V(K)$ and a set $K$ of non-empty finite subsets of $V(K)$, called \emph{simplices}, satisfying that if $\sigma\in K$ and $\tau\subset\sigma$, then $\tau\in K$, and written $\sigma=\langle v_0,\ldots v_s\rangle$. A \emph{simplicial map} $g:K\rightarrow L$ between simplicial complexes is a function $g:V(K)\rightarrow V(L)$ sending simplices to simplices. We shall use two important complexes here. Given any topological space $X$ and a covering of it $U=\{U_{\alpha}\}_{\alpha\in A}$, we can construct the \emph{nerve} of the covering $\mathcal{N}_U(X)$, whose vertices are the elements of the covering and a finite set of members of the covering $\langle U_0,\ldots,U_s\rangle$ is a $s$-simplex if $U_0\cap\ldots U_s\neq\emptyset$. If we consider the case where $X$ is a metric space, and the covering $B_{\varepsilon}=\{\ball{x}{\varepsilon}:x\in X\}$, for $\varepsilon>0$, the nerve $\check{C}_{\varepsilon}(X)$ of this covering is sometimes called the \emph{\v{C}ech} complex. Given a metric space $X$, we define the \emph{Vietoris} (or \emph{Rips}) complex $\mathcal{V}_{\varepsilon}(X)$, for $\varepsilon>0$, as the simplicial complex having as vertices the points of $X$ and as simplices the finite sets $\langle x_0,\ldots,x_s\rangle$ such that $\diam\{x_0,\ldots,x_s\}<\varepsilon$. Given a simplicial complex $K$, its \emph{geometric realization} $|K|$ is the union of simplices of $K$, as a subspace of $\mathbb{R}^N$, and topologized defining as closed sets, the sets meeting each simplex in a closed subset. If $K$ is finite, then this topology is inherited as a subspace of $\mathbb{R}^N$ and, in this case, $|K|$ becomes a compact metric space. A topological space $X$ is a called a \emph{polyhedron} if there exists a simplicial complex $K$ such that $X=|K|$. If we have a simplicial map $g:K\rightarrow L$, the \emph{realization of the map} $g$ is the continuous map $|g|:|K|\rightarrow |L|$ defined linearly. If $g$ is an isomorphism (that is, a bijection on vertices and simplices) then $|g|$ is an homeomorphism. We say that two continuous maps $f,g:X\rightarrow|K|$ are \emph{contiguous} if, for every $x\in X$, $f(x)\cup g(x)$ belongs to the closure of a simplex $\sigma$ of $K$. This is an importan property since contiguous maps are homotopic. The \emph{barycentric subdivision} $K'$ of a simplicial complex $K$ is the simplicial complex whose vertices are the simplices of $K$ and its simplices are finite chains of simplices $\langle\sigma_0,\ldots,\sigma_s\rangle$ satisfying $\sigma_0\subset\ldots\subset\sigma_s$. We shall use the following facts later.
\begin{prop}
Any simplicial map $\xi:K'\rightarrow K$ sending a vertex $\sigma$ of $K'$ to any vertex of $\sigma$ in $K$ is contiguous (and hence homotopic) to the identity. Any simplicial map $g:K\rightarrow L$ induces a subdivided simplicial map $g':K'\rightarrow L'$, with $g'(\sigma)=g(\sigma)$, whose realization is contiguous (and hence homotopic) to $|g|$.
\end{prop}

\paragraph{\textsc{Finite spaces}}
See \cite{Mfinitetopological, Mfinitecomplexes} for details. Any point $x$ of an Alexandroff space $X$ has an open \emph{minimal neighborhood}, namely the intersection of all the open sets containing it, and the set of minimal neighborhoods of points of $X$ are a basis for its topology. This minimal basis defines a reflexive and transitive relation on the space $X$: for $x,y\in X$, say $x\leqslant y$ if $B_x\subset B_y$. This relation is a partial order if and only if $X$ is $T_0$. On the other hand, every reflexive and transitive relation on a set $X$ determines an Alexandroff topology, with basis the sets $U_x=\{y\in X:y\leqslant x\}$. So, for every set, its Alexandroff topologies are in bijective correspondence with its reflexive and transitive relations. The topology is $T_0$ if and only if the relation is a partial order (poset). A function $f:X\rightarrow Y$ of Alexandroff spaces is continuous if and only if is order preserving, that is, $x\leqslant y$ implies $f(x)\leqslant f(y)$. McCord \cite{Msingular} showed that there is a bijective correspondence between simplicial complexes and Alexandroff $T_0$ spaces. Given an Alexandroff space space $X$, define $\mathcal{K}(X)$ as the abstract simplicial complex having has vertex set $X$ and as simplices the finite totally ordered subsets $x_0\leqslant\ldots\leqslant x_s$ of the poset $X$. A continuous map $f:X\rightarrow Y$ of Alexandroff spaces defines a simplicial map $\mathcal{K}(f):\mathcal{K}(X)\rightarrow\mathcal{K}(Y)$, since it is order preserving. Now, we can define the following map $\psi=\psi_X:|\mathcal{K}(X)|\rightarrow X$ as follows. Every point $z\in|\mathcal{K}(X)|$ is contained in the interior of a unique simplex $\sigma$ spanned by a strictly increasing finite sequence $x_0<x_1<\ldots<x_s$ of points of $X$. We define $\psi(z)=x_0$, and the following theorem holds.
\begin{teo}[Alexandroff-McCord Correspondence \cite{Msingular}]\label{teo:AMcorrespondence}
	The map $\psi_X$ is a weak homotopy equivalence. Moreover, given a map $f:X\rightarrow Y$ of A-spaces, the induced simplicial map $\mathcal{K}(f)$ makes the following diagram commutative.
	$$\xymatrix@C=2cm@R=1.5cm{
		X\ar[r]^{f}			 						  											& Y\\
		|\mathcal{K}(X)|\ar[u]^{\psi_X}\ar[r]_{|\mathcal{K}(f)|}		    & |\mathcal{K}(Y)|\ar[u]_{\psi_Y}
	}$$
\end{teo}

\paragraph{\textsc{Hyperspaces}} For the study of general hyperspaces, we recommend the paper \cite{Mtopologies} and the book \cite{Nhyperspaces}. Given a topological space $X$, we define the \emph{hyperspace} of $X$ as the set of its non-empty closed subsets $2^X$, endowed with some topology. . For example, if we have a compact metric space $\cms$, the \emph{Hausdorff distance} makes $2^X$ a compact metric space with some nice properties relating $X$ and $2^X$ \cite{MGthe}. The topology for $2^X$ used through this paper is the \emph{upper semifinite topology}, with a basis consisting of the open sets $$B(U)=\conjunto{C\in 2^X: C\subset U}\subset 2^X$$
for every $U$ open in $X$. We write $2^X_u$ for it. In general, this is a non-Hausdorff topology (see \cite{MGupper,MGhomotopical} for more details). Let $\cms$ be a compact metric space. Consider for every $\varepsilon>0$ the subspace of $2^X$ consisting of the closed subsets of $X$
$$U_{\varepsilon}=\conjunto{C\in 2^X:\diam(C)<\varepsilon}.$$ The following result from \cite{MGhomotopical} is key in the use of the upper semifinite topology for hyperspaces in this text.
\begin{prop}
The family $U=\{U_{\varepsilon}\}_{\varepsilon>0}$ is a base of open neighborhoods of the canonical copy of $X$ inside $2^X_u$ (that is, as the inclusion of every point as a singleton).
\end{prop}

\paragraph{\textsc{Shape theory}} The approach used here will be the Inverse System, initiated and by Mardesic and Segal for compact Hausdorff spaces in \cite{MSshapes} and developed by them and others. Refer to \cite{MSshape} for details. Let $\mathcal{C}$ be any category and $\Lambda$ be a directed set (called the \emph{index set}). An \emph{inverse system} in $\mathcal{C}$ consists of a triple $\mathbf{X}=\left(X_{\lambda},p_{\lambda\lambda'},\Lambda\right)$, where $X_{\lambda}$ is an object (\emph{term}) of $\mathcal{C}$, for every $\lambda\in\Lambda$, and $p_{\lambda\lambda'}:X_{\lambda'}\rightarrow X_{\lambda}$ is a morphism (\emph{bonding morphisms or maps}) of $\mathcal{C}$, for every pair $\lambda\leqslant\lambda'$ of indices, satisfiying $p_{\lambda\lambda}=id_{X_{\lambda}}$ and $p_{\lambda\lambda'}p_{\lambda'\lambda''}=p_{\lambda\lambda''}$, for every triple $\lambda\leqslant\lambda'\leqslant\lambda''$. If the index set of an inverse system is $\Lambda=\mathbb{N}$, then it is called \emph{inverse sequence}, and it is written $\mathbf{X}=\left\lbrace X_n,p_{nn+1}\right\rbrace$, since the rest of bonding maps are determined by the composition of those. Given two inverse systems $\mathbf{X}=\left(X_{\lambda},p_{\lambda\lambda'},\Lambda\right)$, $\mathbf{Y}=\left(Y_{\mu},q_{\mu\mu'},\Lambda\right)$, a \emph{level morphism of systems} is a collection of morphisms of $\mathcal{C}$, for every $\mu\in M$, $f_{\lambda}:X_{\lambda)}\rightarrow Y_{\lambda}$, for every $\lambda\leqslant\lambda'$, the following diagram is commutative.
$$\xymatrix@C=2cm@R=1.5cm{
	X_{\lambda}\ar[d]_{f_{\lambda}}& X_{\lambda'}\ar[l]\ar[d]^{f_{\lambda'}}\\
	Y_{\lambda} & Y_{\lambda'}\ar[l]}$$
The composition of these morphisms is defined straightforward. The category \emph{pro-$\mathcal{C}$} has objects inverse systems $\mathbf{X}$ (over all directed sets) in $\mathcal{C}$ and morphisms $\mathbf{f}:\mathbf{X}\rightarrow\mathbf{Y}$, equivalence classes of morphisms of systems under an equivalence relation (we do not define the relaation and the more general classes of morphisms because they are not needed here). If $\Lambda'$ is a cofinal subset of $\Lambda$, that is, for every $\lambda'\in\Lambda'$, there is a $\lambda\in\Lambda$ with $\lambda<\lambda'$, we can use $\Lambda'$ instead $\Lambda$ and get an isomorphic inverse system. Next, we state a very useful characterization about isomorphisms in pro-$\mathcal{C}$.
\begin{teo}[Morita's lemma \cite{Mthe}]\label{teo:morita}
A level morphism of systems $$\mathbf{f}:\mathbf{X}=\left(X_{\lambda},p_{\lambda\lambda'},\Lambda\right)\longrightarrow\mathbf{Y}=\left(Y_{\lambda},q_{\lambda\lambda'},\Lambda\right)$$ in pro-$\mathcal{C}$, is an isomorphism if and only if every $\lambda\in\Lambda$ admits a $\lambda'\geqslant\lambda$ and a morphism $g_{\lambda}:Y_{\lambda'}\rightarrow X_{\lambda}$ in $\mathcal{C}$ making the following diagram commutative.
$$\xymatrix@C=2cm@R=1.5cm{
X_{\lambda}\ar[d]_{f_{\lambda}} & X_{\lambda'}\ar[l]\ar[d]^{f_{\lambda'}}\\
Y_{\lambda}								    & Y_{\lambda'}\ar[l]\ar@{-->}[ul]_{g_{\lambda}}}$$
\end{teo}
Let $\mathcal{T}$ be a category and $\mathcal{P}$ a subcategory of $\mathcal{T}$. Let $X$ be an object of $\mathcal{T}$, a \emph{$\mathcal{T}$-expansion} of $X$ is a morphism  in pro-$\mathcal{T}$ $\mathbf{p}:X\rightarrow\mathbf{X}$ (consider $X$ as an inverse system in which every term is $X$ and the bonding maps are the identity) to an inverse system $\mathbf{X}=(X_{\lambda},p_{\lambda\lambda'},\Lambda)$ in $\mathcal{T}$ satisfying a universal condition (not included). Moreover, $\mathbf{p}$ is called a \emph{$\mathcal{P}$-expansion} of $X$ whenever $\mathbf{X}$ and $\mathbf{f}$ are in pro-$\mathcal{P}$. Every two expansions of the same space are isomorphic. Also, a composition of an expansion with an isomophism is again an expansion. Given a category $\mathcal{T}$ and a subcategory $\mathcal{P}$, we say that $\mathcal{P}$ is \emph{dense} in $\mathcal{T}$ if every object of $\mathcal{T}$ admits a $\mathcal{P}$-expansion, under an equivalence (also not included). Define the \emph{shape category} $Sh$ for $(\mathcal{T},\mathcal{P})$ as the one having as objects the objects of $\mathcal{T}$, and, for $X,Y\in\mathcal{T}$, the morphisms $X\rightarrow Y$ are the equivalence classes of morphisms $\mathbf{f}:\mathbf{X}\rightarrow\mathbf{Y}$ in pro-$\mathcal{P}$. Usually, the shape category is used for $(\mathcal{T}=HTop,\mathcal{P}=HPol)$. Shape is an extension of homotopy, so two spaces homotopically equivalent must have the same shape. Two isomorphic spaces $X,Y$ in $Sh$ are said to have the \emph{same shape (type)}, written $Sh(X)=Sh(Y)$. There are some well known HPol-expansions defined for any topological space (such as the Vietoris or the \v{C}ech expansions). Inverse limits are a good source of HPol-expansions, since they and inverse system considered in HPol is an HPol-expansion of its inverse limit. The Shape category has its own invariants, defined through the HPol-expansions and its inverse limits. For example, for every abelian group $G$, we can consider the $k$-th homology group $H_k(X_{\lambda};G)$ of each term and the induced homology maps $H_k(p_{\lambda\lambda'};G)$ of the bonding maps. Then we obtain an inverse system of abelian groups $$H_k(\mathbf{X};G)=\left(H_k(X_{\lambda};G),H_k(p_{\lambda\lambda'};G),\Lambda\right),$$ called the \emph{$k$-th homology pro-group} of $X$. We define the \emph{$k$-th \v{C}ech homology group} of $X$ as the inverse limit of this inverse system of groups, $$\check{H}_k(X)=\lim_{\leftarrow}H_k(\mathbf{X};G).$$ Similarly, we can obtain the \emph{$K$-th shape group} of $X$, as the inverse limit of the corresponding homotopy pro-groups $\pi_k(\mathbf{X})$. It is known that the \v{C}ech homology and shape groups are well defined, that is, they do not depend on the HPol-expansion we use to compute them. Movability is another invariant, trying to generalize the concept of spaces having the shape of ANRs. We can define movability for arbitrary inverse systems. An inverse system $\mathbf{X}=(X_{\lambda},p_{\lambda\lambda'},\Lambda)$ in pro-$\mathcal{C}$ is \emph{movable} provided every $\lambda\in\Lambda$ admits $\lambda'\geqslant\lambda$, called \emph{movability index} of $\lambda$, such that, for every $\lambda''\geqslant\lambda$, there exists a morphism $r:X_{\lambda'}\rightarrow X_{\lambda''}$ of $\mathcal{C}$ such that $p_{\lambda\lambda'}=p_{\lambda\lambda''} r$. A topological space $X$ is \emph{movable} if it has movable HPol-expansions. An inverse system $\mathbf{X}$ in pro-$\mathcal{C}$ is \emph{stable} provided it is isomorphic in pro-$\mathcal{C}$ to an object $X\in\mathcal{C}$. It is evident that if $\mathbf{X}$ is stable, then it is movable. A topological space is said to be \emph{stable} provided its HPol-expansions are stable. It is equivalent to have the same shape as a polyhedron (or ANR) and it is obviously a shape invariant property. Then, an stable space is movable, but the converse is not always true. An inverse system of groups $\mathbf{G}=(G_{\lambda},p_{\lambda\lambda'},\Lambda)$ has the \emph{Mittag-Leffler} (\textsc{ml}) property if every $\lambda\in\Lambda$ admits a $\lambda'\geqslant\lambda$ (called an \emph{\textsc{ml} index} for $\lambda$) such that, for every $\lambda''\geqslant\lambda'$, we have $p_{\lambda\lambda''}(X_{\lambda''})=p_{\lambda\lambda'}(X_{\lambda})$. Using that every movable inverse system of groups has the Mittag-Leffler property and that every functor preserves the property of being movable, we have that the homology and homotopy pro-groups $H_k(\mathbf{X};G)$, $\pi_k(\mathbf{X})$ are movable and hence have the Mittag-Leffler property.

\paragraph{\textsc{The Main Construction}}
We recall the most important tool of this text, the \emph{Main Construction}, introduced in \cite{MCLepsilon} and reformulated in \cite{MMreconstruction}. A sequence of finite approximations for a compact metric space $X$, with some special properties, is constructed. Let $\cms$ be a compact metric space. For every $\varepsilon>0$, we can find a finite \emph{$\varepsilon$-approximation} $A\subset X$, that is, for every point $x\in X$, there is a point $a\in A$ with $\dist{x,a}<\varepsilon$. An \emph{adjusted approximative sequence} $\left\lbrace\varepsilon_n,A_n\right\rbrace$ consists of a decreasing and tending to zero sequence of positive real numbers $\{\varepsilon_n\}$ and a sequence of finite spaces $\{A_n\}$ such that, for every $\todon$, $A_n$ is an $\varepsilon_n$-approximation of $X$ and $\varepsilon_{n+1}$ is adjusted to it, meaning that, for every $\todon$, $\varepsilon_{n+1}<\frac{\varepsilon_{n}-\gamma_{n}}{2}$, with \[\gamma_n=\text{sup}\{\text{d}(x,A_n):x\in X\}<\varepsilon_n.\] The Main Construction in \cite{MCLepsilon} and Theorem 3 in \cite{MMreconstruction} states the existence of adjusted approximative sequences for every compact metric space. Given an adjusted approximative sequence, we can define a \emph{nearby approximative sequence} $\{q_{A_n}\}$ as the sequence of continuous maps $q_{A_n}:X\rightarrow2^{A_n}_u$, defined by proximity, \[q_{A_n}(x)=\left\lbrace a\in A:\text{d}(x,A_n)=\text{d}(x,a)\right\rbrace,\] and a \emph{finite approximative sequence} (\textsc{fas} for short) $\left\lbrace U_{2\varepsilon_n}(A_{n}),p_{n,n+1}\right\rbrace$, as an inverse sequence with terms \[U_{2\varepsilon_n}(A_{n}=\conjunto{C\in 2^{A_n}_u:\diam(C)<\varepsilon_n}\] and continuous maps \[p_{n,n+1}:U_{2\varepsilon_{n+1}}(A_{n+1})\longrightarrow U_{2\varepsilon_{n}}(A_{n})\] with $p_{n,n+1}=q_{A_n}(C)$, for every closed set $C\in U_{n+1}(A_{n+1})$. Note that the upper semifinite topology in the hyperspace of the discrete space $U_{2\varepsilon_{n}}(A_{n}$ is a finite space with the relation $\subset$ as poset. 
In \cite{MMreconstruction} it is shown that the inverse limit of every \textsc{fas} has the homotopy type of the original space, since it contains an homeomorphic copy of it which is an strong deformation retract of this limit. The General Principle of \cite{MCLepsilon} states that the associated Alexandroff-McCord sequence of polyhedra $\left\lbrace |\mathcal{K}U_{2\varepsilon_n}(A_{n})|,|\mathcal{K}(p_{n,n+1})|\right\rbrace$ (using the correspondence of Theorem \ref{teo:AMcorrespondence}) is able to capture the shape properties of $X$. We give a positive answer in Section \ref{sec:GP} showing that this inverse sequence is indeed an HPol-expansion for $X$ and we generalize the General Principle giving some other HPol-expansions of $X$ using the Main Construction. 

\section{Polyhedral approximative sequences and the General Principle}\label{sec:GP}
In this section we show that some inverse sequences of polyhedra based on the Main Construction over a compact metric space are Hpol-expansions and hence they provide all the shape information about the space. These sequences are also useful for inverse persistence computations.

Consider an adjusted sequence $\{\varepsilon_n A_n\}$ and a \textsc{fas} $\left\lbrace U_{2\varepsilon_n}(A_{n}),p_{n,n+1}\right\rbrace$. Using the Alexandroff-McCord correspondence (Theorem \ref{teo:AMcorrespondence}) we have that, for every $n\in\mathbb{N}$ and finite $\textrm{T}_0$ space $U_{2\varepsilon_n}(A_n)$, there exists a simplicial complex $\mathcal{K}(U_{2\varepsilon_n}(A_n))$ with vertex set the points $D\in U_{2\varepsilon_n}(A_n)$ and simplexes $\langle D_0,D_1,\ldots,D_s\rangle$ with $D_0\subset D_1\subset\ldots\subset D_s$ such that there is a weak homotopy equivalence between the finite space and the geometric realization of the simplicial complex $$f_n:|\mathcal{K}(U_{2\varepsilon_n}(A_n))|\longrightarrow U_{2\varepsilon_n}(A_n),$$ defined as follows. Every point $x\in|\mathcal{K}(U_{2\epsilon_n}(A_n))|$ is contained in the interior of a unique simplex $\sigma=\langle D_0,D_1,\ldots,D_s\rangle$ and $f_n(x)=D_0$. We also have simplicial maps (following McCords paper's notation we should write $\mathcal{K}(p_{n,n+1})$ for the simplicial maps but we will omit this notation, using the same as for the maps between the finite spaces, $p_{n,n+1}$, for the sake of simplicity) between the polyhedra, defined on the vertices and extended as usual to simplices
\begin{eqnarray*}
	p_{n,n+1}:\mathcal{K}(U_{2\epsilon_{n+1}}(A_{n+1})) & \longrightarrow & \mathcal{K}(U_{2\epsilon_n}(A_n)) \\
	D & \longmapsto & p_{n,n+1}(D) \\
	\langle D_0, D_1, \ldots, D_s\rangle & \longmapsto & \langle p_{n,n+1}(D_0), p_{n,n+1}(D_1), \ldots, p_{n,n+1}(D_s)\rangle
\end{eqnarray*} 
where, if $$D_0\subset D_1\subset\ldots\subset D_s,$$ then $$p_{n,n+1}(D_0)\subset p_{n,n+1}(D_1)\subset\ldots\subset p_{n,n+1}(D_s).$$ The realizations of these simplicial maps satisfy that, for every $n\in\mathbb{N}$, the diagram 
$$\xymatrix@C=2cm@R=1.5cm{|\mathcal{K}(U_{2\varepsilon_n}(A_n))|\ar[d]_{f_n} & |\mathcal{K}(U_{2\varepsilon_{n+1}}(A_{n+1}))|\ar[d]^{f_{n+1}}\ar[l]_{|p_{n,n+1}|}\\
	U_{2\varepsilon_n}(A_n) & U_{2\varepsilon_{n+1}}(A_{n+1})\ar[l]^{p_{n,n+1}}}$$ 
commutes. So we obtain an inverse sequence of polyhedra and a map (actually a level map) between the inverse sequences of finite spaces and polyhedra.
\begingroup\makeatletter\def\f@size{10}\check@mathfonts
$$\xymatrix{|\mathcal{K}(U_{2\varepsilon_1}(A_1))|\ar[d]_{f_1} & |\mathcal{K}(U_{2\varepsilon_{2}}(A_{2}))|\ar[l]_{|p_{1,2}|}\ar[d]_{f_2}&\ldots\ar[l]&|\mathcal{K}(U_{2\varepsilon_n}(A_n))|\ar[d]_{f_n}\ar[l] & |\mathcal{K}(U_{2\varepsilon_{n+1}}(A_{n+1}))|\ar[d]^{f_{n+1}}\ar[l]_{|p_{n,n+1}|}&\ldots\ar[l]\\
	U_{2\varepsilon_1}(A_1) & U_{2\varepsilon_{2}}(A_{2})\ar[l]^{p_{1,2}}&\ldots\ar[l]&U_{2\varepsilon_n}(A_n)\ar[l] & U_{2\varepsilon_{n+1}}(A_{n+1})\ar[l]^{p_{n,n+1}}&\ldots\ar[l]}$$ 
\endgroup
The analysis of these two inverse sequences, their limits and their relations with the original space are proposed in \cite{MCLepsilon}. We will show here that this and other related polyhedral inverse sequences are HPol-expansions of the space $X$.

\paragraph{The Alexandroff-McCord Approximative Sequence}
We begin by considering the inverse sequence of polyhedra mentioned above, $$\mathcal{K}(X)=\{|\mathcal{K}(U_{2\varepsilon_n}(A_n))|, |p_{n,n+1}|\},$$ which will be called the \emph{Alexandroff-McCord Approximative Sequence}. We show that this inverse sequence defines the shape of $X$. It is closely related with a well known HPol-expansion that uses the Vietoris-Rips complexes. The \emph{Vietoris-Rips complex} $\mathcal{R}_{\varepsilon}(X)$ of a compact metric space $(X,\textrm{d})$ for $\varepsilon>0$ is the simplicial complex with vertex set $X$ and a $q$-simplex is a subset $\left\langle x_0,\ldots, x_q\right\rangle\subset X$ such that $\diam\{x_0,\ldots, x_q\}<\varepsilon$. As stated in Corollary 7 of \cite{MCLepsilon}, for a finite metric space $(A,\textrm{d})$ and $\varepsilon>0$, we have $\mathcal{K}(U_{\varepsilon}(A))=\mathcal{R}'_{\varepsilon}(A)$.
\begin{prop}\label{teo:mccordexpansion}
The Alexandroff-McCord Aproximative Sequence $\mathcal{K}(X)$ is an HPol-expansion of $X$.
\end{prop} 
\begin{proof}
We shall show that $\mathcal{K}(X)$ is isomorphic to the Vietoris system $$\mathcal{V}(X)=\{|\mathcal{R}_{\varepsilon}(X)|,|i_{\varepsilon,\varepsilon'}|,\mathbb{R}\},$$ where, for every $\varepsilon\leqslant\varepsilon'$, $$i_{\epsilon,\epsilon'}:|\mathcal{R}_{\varepsilon'}(X)|\longrightarrow|\mathcal{R}_{\varepsilon}(X)|$$ is just the simplicial inclusion. This is a well known HPol-expansion of $X$ (see \cite{MSshape}) but, instead, let us consider the cofinal (and hence isomorphic) sequence (keeping the notation for the sake of simplicity) $$\mathcal{V}(X)=\{|\mathcal{R}_{2\epsilon_n}(X)|,|i_{2\epsilon_n,2\epsilon_{n+1}}|\}.$$ 
We define an intermediate sequence in order to set this isomorphism. Consider, for all $n\in\mathbb{N}$, the map
\begin{eqnarray*}
p^*_{n,n+1}:\mathcal{R}_{2\epsilon_{n+1}}(A_{n+1}) & \longrightarrow & \mathcal{R}_{2\epsilon_n}(A_n)\\
a & \longmapsto & b\in q_{A_n}(a)=p_{n,n+1}(\{a\}),
\end{eqnarray*}
where $b$ is any of the elements of $q_{A_n}(a)$. It is a simplicial map since a simplex $\langle a_0,a_1,\ldots,a_s\rangle\in\mathcal{R}_{2\epsilon_{n+1}}(A_{n+1})$ has image $\langle b_0,b_1,\ldots,b_s\rangle$ with $b_i\in q_{A_n}(a_i)$ for $i=0,1,\ldots,s$, which is a simplex of $\mathcal{R}_{2\epsilon_n}(A_n)$, because, for every $i,j\in\{0,1,\ldots,s\}$, \[\textrm{d}(b_i,b_j)\leq\textrm{d}(b_i,a_i)+\textrm{d}(a_i,a_j)+\textrm{d}(a_j,b_j)<2\gamma_n+2\epsilon_{n+1}<2\epsilon_n.\] The realization of this map on the corresponding polyhedra is well defined up to homotopy type (we do not define a map, but a homotopy class of maps) because for $b,b'\in q_{A_n}(a)$, with $a\in A_{n+1}$, we have $$\textrm{d}(b,b')\leqslant \textrm{d}(b,a)+\textrm{d}(a,b')<2\varepsilon_n.$$ 
We can define inductively the inverse HPol-sequence of polyhedra $$\mathcal{M}(X)=\{|\mathcal{R}_{2\epsilon_n}(A_n)|,|p^*_{n,n+1}|\}_{\mathbb{N}}.$$ 
We next show $\mathcal{M}(X)$ is isomorphic to $\mathcal{K}(X)$ using the canonical simplicial map $\rho:K'\rightarrow K$, for a simplicial complex $K$, defined as $\rho(\{x_0,x_1,\ldots,x_s\})=x_s$ for a vertex $\{x_0,x_1,\ldots,x_s\}$ of $K'$ and sending a simplex $\langle\{x_0\},\{x_0,x_1\},\ldots,\{x_0,x_1,\ldots,x_s\}\rangle$ of $K'$ to the simplex $\langle x_0,x_1,\ldots,x_s\rangle$. The realization of this map $|\rho|:|K'|=|K|\rightarrow |K|$ is homotopic to the identity (we could have defined the simplicial map choosing any vertex as image and the realization would have the same homotopic properties, see \cite{Mfinitecomplexes}). This map induces componentwise a morphism of systems $\rho:\mathcal{K}(X)\rightarrow \mathcal{M}(X)$ through $\rho_n:\mathcal{R}'_{2\epsilon_n}(A_n)\rightarrow \mathcal{R}_{2\epsilon_n}(A_n)$ because, for every $n\in\mathbb{N}$, the following diagram is commutative up to homotopy.
$$\xymatrix @C=2cm@R=1.5cm{|\mathcal{R}'_{2\epsilon_n}(A_n)|\ar[d]_{|\rho_n|} & |\mathcal{R}'_{2\epsilon_{n+1}}(A_{n+1})|\ar[l]_{|p_{n,n+1}|}\ar[d]^{|\rho_{n+1}|}\\
 |\mathcal{R}_{2\epsilon_n}(A_n)| & |\mathcal{R}_{2\epsilon_{n+1}}(A_{n+1})|\ar[l]^{|p^*_{n,n+1}|}}$$ 
If $x\in|\mathcal{R}'_{2\epsilon_{n+1}}(A_{n+1})|$, then it belongs to some simplex $\sigma=\langle D_0,D_1,\ldots,D_s\rangle\in\mathcal{R}'_{2\epsilon_{n+1}}(A_{n+1})$. We write
$$D_s=\underbrace{\underbrace{\underbrace{a_0^0,a_1^0,\ldots,a^0_{r_0}}_{D_0},a^1_0,a^1_1,\ldots,a^1_{r_1}}_{D_1},\ldots}_{D_{s-1}},a_0^s,a_1^s,\ldots,a_{r_s}^s,$$
$$p_{n,n+1}(D_s)=\underbrace{\underbrace{\underbrace{q_{A_n}(a_0^0)\cup\ldots\cup q_{A_n}(a^0_{r_0})}_{p_{n,n+1}(D_0)}\cup q_{A_n}(a^1_0)\cup \ldots\cup q_{A_n}(a^1_{r_1})}_{p_{n,n+1}(D_1)}\cup \ldots}_{p_{n,n+1}(D_{s-1})}\cup q_{A_n}(a_0^s)\cup \ldots\cup q_{A_n}(a_{r_s}^s),$$
where $q_{A_n}(a_{r_i}^i)=\{b_0,b_1,\ldots b_{t_i}\}$ for $i=0,1,\ldots,s$. Then,
\begin{eqnarray*}
 & \langle D_0,D_1,\ldots,D_s\rangle & \xmapsto{\enspace p_{n,n+1}\enspace}\langle p_{n,n+1}(D_0),p_{n,n+1}(D_1),\ldots,p_{n,n+1}(D_s)\rangle\xmapsto{\enspace\rho_n\enspace}\langle b_{t_0}^0,b_{t_1}^1,\ldots,b_{t_s}^s\rangle=\sigma_1,\\
 & \langle D_0,D_1,\ldots,D_s\rangle & \xmapsto{\enspace\rho_{n+1}\enspace}\langle a_{r_0}^0,a_{r_1}^1,\ldots,a_{r_s}^s\rangle\xmapsto{\enspace p_{n,n+1}^*\enspace}\langle b_{j_0}^0,b_{j_1}^1,\ldots,b_{j_s}^s\rangle=\sigma_2, 
\end{eqnarray*}
for $j_i\in\{0,1,\ldots,t_i\}$ with $i=0,1,\ldots,s$. But $\sigma_1$ and $\sigma_2$ lie in a common simplex $\sigma_1\cup\sigma_2=\langle b_{t_0}^0,b_{t_1}^1,\ldots,b_{t_s}^s,b_{j_0}^0,b_{j_1}^1,\ldots,b_{j_s}^s\rangle$ of $\mathcal{R}'_{2\epsilon_n}(A_n)$ because, for any pair $u,v\in\{0,1,\ldots,s\}$, it is satisfied that
\[\textrm{d}(b^u_{t_u},b^v_{j_v})\leq\textrm{d}(b^u_{t_u},a_{r_u}^u)+\textrm{d}(a_{r_u}^u,a_{r_v}^v)+\textrm{d}(a^v_{r_v},b^v_{j_v})\leq 2\gamma_n+2\varepsilon_{n+1}<2\epsilon_n.\]
So, $\rho:\mathcal{K}(X)\rightarrow \mathcal{M}(X)$ is a morphism of systems. Moreover, $\rho$ is equivalent to the identity as morphism of systems, since the equivalence relation (page 6 of \cite{MSshape}) is trivially satisfied because, for all $n\in\mathbb{N}$, $$\rho,\textrm{id}:|\mathcal{R}'_{2\epsilon_n}(A_n)|\longrightarrow|\mathcal{R}_{2\epsilon_n}(A_n)|$$
are homotopic maps, as we pointed out. Then $\mathcal{K}(X)$ and $\mathcal{M}(X)$ are isomorphic.

It just remains to prove that the inverse systems $\mathcal{V}(X)$ and $\mathcal{M}(X)$ are isomorphic. Both systems deal with Vietoris Rips complexes (not barycentric subdivisions), so we consider, for each $n\in\mathbb{N}$, the obvious inclusion $j_n:\mathcal{R}_{2\epsilon_n}(A_n)\rightarrow\mathcal{R}_{2\epsilon_n}(X)$, which is a simplicial map and defines a morphism of systems $j:\mathcal{M}(X)\rightarrow \mathcal{V}(X)$ since, for every $n\in\mathbb{N}$, the diagram
$$\xymatrix@C=2cm@R=1.5cm{|\mathcal{R}_{2\epsilon_n}(A_n)|\ar[d]_{|j_n|} & |\mathcal{R}_{2\epsilon_{n+1}}(A_{n+1})|\ar[d]^{|j_{n+1}|}\ar[l]_{|p^*_{n,n+1}|}\\
|\mathcal{R}_{2\epsilon_n}(X)| & |\mathcal{R}_{2\epsilon_{n+1}}(X)|\ar[l]^{|i_{2\epsilon_n,2\epsilon_{n+1}}|}}$$ commutes, up to homotopy. Indeed, every $x\in|\mathcal{R}_{2\epsilon_n}(A_n)|$ belongs to a simplex $\sigma=\langle a_0,a_1,\ldots,a_s\rangle\in\mathcal{R}_{2\epsilon_n}(A_n)$. The images of the simplex are
\begin{eqnarray*}
 & \langle a_0,a_1,\ldots,a_s\rangle & \xmapsto{\enspace p^*_{n,n+1}\enspace}\langle b_0,b_1,\ldots,b_s\rangle\xmapsto{\enspace j_n\enspace}\langle b_0,b_1,\ldots,b_s\rangle=\sigma_1,\\
 & \langle a_0,a_1,\ldots,a_s\rangle &\xmapsto{\enspace j_{n+1}\enspace}\langle a_0,a_1,\ldots,a_s\rangle\xmapsto{\enspace i_{2\epsilon_n,2\epsilon_{n+1}}\enspace}\langle a_0,a_1,\ldots,a_s\rangle=\sigma_2,
\end{eqnarray*}
with $\enspace b_i\in q_{A_n}(a_i)$ for $i=0,1,\ldots s$, which lie on a common simplex $\sigma_1\cup\sigma_2=\langle b_0,b_1,\ldots,b_s,a_0,a_1,\ldots,a_s\rangle$ because, for every $i,j\in\{0,1,\ldots,s\}$ we have
\[\textrm{d}(a_i,b_j)\leq\textrm{d}(a_i,a_j)+\textrm{d}(a_j,b_j)< 2\epsilon_{n+1}+\gamma_n=\epsilon_n\]
which means $\sigma_1\cup\sigma_2\in\mathcal{R}_{2\epsilon_n}(X)$ and hence the maps are homotopic.
This morphism of systems is, in fact, an isomorphism. To see this, we need to use Morita's Lemma (Theorem \ref{teo:morita}). Roughly speaking, all we need is a diagonal map making the last diagram factorizing through it. So, we define, for every $n\in\mathbb{N}$, a simplicial map
\begin{eqnarray*}
g_n:\mathcal{R}_{2\epsilon_{n+1}}(X) & \longrightarrow & \mathcal{R}_{2\epsilon_n}(A_n)\\
x & \longmapsto & a\in q_{A_n}(x),
\end{eqnarray*}
which is simplicial because if $g_n(\langle x_0,x_1,\ldots,x_s\rangle)=\langle a_0,a_1,\ldots,a_s\rangle$, then, for every $i,j\in\{0,1,\ldots,s\}$, we have
\[\textrm{d}(a_i,a_j)\leq\textrm{d}(a_i,x_i)+\textrm{d}(x_i,x_j)+\textrm{d}(x_j,a_j)<2\gamma_n+2\epsilon_{n+1}<2\epsilon_n,\]
and induces a well defined homotopy class of maps since $a,a'\in q_{A_n}(x)$ implies that $\textrm{d}(a,a')<2\epsilon_n$. The realization of this simplicial map is our diagonal that makes the diagram conmutative up to homotopy.
$$\xymatrix@C=2cm@R=1.5cm{|\mathcal{R}_{2\epsilon_n}(A_n)|\ar[d]_{|j_n|} & |\mathcal{R}_{2\epsilon_{n+1}}(A_{n+1})|\ar[d]^{|j_{n+1}|}\ar[l]_{|p^*_{n,n+1}|}\\
|\mathcal{R}_{2\epsilon_n}(X)| & |\mathcal{R}_{2\epsilon_{n+1}}(X)|\ar[l]^{|i_{2\epsilon_n,2\epsilon_{n+1}}|}\ar[ul]_{|g_n|}}$$
The up-right subdiagram commutes because if
\begin{eqnarray*}
 & \langle a_0,a_1,\ldots,a_s\rangle & \xmapsto{\enspace p^*_{n,n+1}\enspace}\langle b_0,b_1,\ldots,b_s\rangle,\\
 & \langle a_0,a_1,\ldots,a_s\rangle &  \xmapsto{\enspace j_{n+1}\enspace}\langle a_0,a_1,\ldots,a_s\rangle \xmapsto{\enspace g_n\enspace}\langle b_0',b_1',\ldots,b_s'\rangle,
\end{eqnarray*}
with $b_i, b_i'\in q_{A_n}(a_i)$, then $\textrm{d}(b_i,b_j')<2\epsilon_n,$ and the two maps are homotopic. Finally, the down-left subdiagram commutes because if we write
\begin{eqnarray*}
 & \langle x_0,x_1,\ldots,x_s\rangle & \xmapsto{\enspace g_n\enspace}\langle a_0,a_1,\ldots,a_s\rangle\xmapsto{\enspace j_n\enspace}\langle a_0,a_1,\ldots,a_s\rangle,\\
 & \langle x_0,x_1,\ldots,x_s\rangle & \xmapsto{\enspace i_{2\epsilon_n,2\epsilon_{n+1}}\enspace}\langle x_0,x_1,\ldots,x_s\rangle,
\end{eqnarray*}
with $a_i\in q_{A_n}(x_i)$, then $\textrm{d}(x_i,a_j)<\epsilon_n$ and we are done.
\end{proof}



\paragraph{The \v{C}ech approximative sequence} We construct an Hpol-expansion of $X$ with nerves of coverings based on our finite approximations. Let us consider the Main Construction on the compact metric space $X$. For every $n\in\mathbb{N}$, consider the set $\textrm{B}_n=\{\textrm{B}_n(a):a\in A_n\}$ of open balls $\textrm{B}_n(a)=\textrm{B}(a,\varepsilon_n)$ with center $a$ and radius $\varepsilon_n$, which is a covering of $X$ since $A_n$ is a $\varepsilon_n$-approximation. We can consider the nerves of these coverings $\mathcal{N}(B_n)$ and define the maps \mapeo{p_{B_n,B_{n+1}}}{\mathcal{N}(B_{n+1})}{\mathcal{N}(B_n)}{B_{n+1}(a)}{B_n(b),\enspace b\in q_{A_n}(a).} This is a simplicial map. Let $\langle\textrm{B}_{n+1}(a_0),\textrm{B}_{n+1}(a_1),\ldots,\textrm{B}_{n+1}(a_s)\rangle$ be a simplex of $\mathcal{N}(B_{n+1})$, so $\textrm{B}_{n+1}(a_0)\cap\textrm{B}_{n+1}(a_1)\cap\ldots\cap\textrm{B}_{n+1}(a_s)\neq\emptyset$. Let $x\in X$ be a point of this intersection, that means $\textrm{d}(x,a_i)<\varepsilon_{n+1}$ for all $i=0,1,\ldots,s$. Let us write the image of this simplex as $\langle\textrm{B}_n(b_0),\textrm{B}_n(b_1),\ldots,\textrm{B}_n(b_s)\rangle$, with $b_i\in q_{A_n}(a_i)$ for all $i=0,1,\ldots,s$. Then, since 
\[\textrm{d}(x,b_i)\leqslant\textrm{d}(x,a_i)+\textrm{d}(a_i,b_i)<\varepsilon_{n+1}+\gamma_n<\varepsilon_n,\]
we obtain $x\in\textrm{B}_n(b_0)\cap\textrm{B}_n(b_1)\cap\ldots\cap\textrm{B}_n(b_s)\neq\emptyset$ and therefore $\langle\textrm{B}_n(b_0),\textrm{B}_n(b_1),\ldots,\textrm{B}_n(b_s)\rangle$ is a simplex of $\mathcal{N}(B_n)$. The realization of these maps $$|p_{B_n,B_{n+1}}|:|\mathcal{N}(B_{n+1})|\longrightarrow|\mathcal{N}(B_n)|$$  are up to homotopy well defined maps, since if $b,b'\in q_{A_n}(a)$, with $a\in A_{n+1}$, then $a\in\textrm{B}_n(b)\cap\textrm{B}_n(b')$ so the two different images are contiguous, and then homotopic, hence, define an Hmap. Thus we obtain the inverse sequence of polyhedra $$\mathcal{\check{C}}_A(X)=\{|\mathcal{N}(B_n)|,|p_{B_n,B_{n+1}}|\}$$ which will be called the \emph{\v{C}ech approximative sequence}. As the Alexandrov-McCord sequence, we see that
\begin{prop}
The \v{C}ech approximative sequence $\mathcal{\check{C}}_A(X)$ is an HPol-expansion of $X$.
\end{prop}
\begin{proof}
To show this, we need to find an isomorphism with another HPol-expansion. We will use the well known \v{C}ech expansion. This is the inverse system in HPol, 
$$\mathcal{\check{C}}(X)=\{|\mathcal{N}(U)|,|g_{U,V}|\}_{\Lambda},$$ where $\Lambda$ is the set of all open coverings of $X$, ordered by refinement, and for every pair of coverings $U,V\in\Lambda$, such that $V$ refines $U$, the Hmaps $|g_{U,V}|$ are the (up to homotopy) realizations of the simplicial maps \mapeo{g_{U,V}}{\mathcal{N}(V)}{\mathcal{N}(U)}{V_{\alpha}}{U_{\alpha}} with $V_{\alpha}\subset U_{\alpha}$. To set the isomorphism with our sequence we use a cofinal sequence of this system more suitable. First of all, we observe that the set of open coverings of $X$ consisting of $\{\textrm{B}_n\}_{\todon}$ are a cofinal directed subset of $\Lambda$. Indeed, if $a\in A_{n+1}$, for every $b\in q_{A_n}(a)$ we have that $\textrm{B}_{n+1}(a)\subset \textrm{B}_{n}(b)$, since for every $c\in\textrm{B}_{n+1}(a)$,
\begin{eqnarray*}\textrm{d}(b,c) & \leqslant & \textrm{d}(b,a)+\textrm{d}(a,c)<\gamma_n+\varepsilon_{n+1}<\varepsilon_n.
\end{eqnarray*} 
That means $\textrm{B}_{n+1}$ refines $\textrm{B}_{n}$ for every $\todon$. So, we can use this new set of indexes to define the isomorphic inverse sequence $$\mathcal{\check{C}}(X)=\{|\mathcal{N}(B_n)|,|g_{B_n,B_{n+1}}|\}.$$
Now, for every $\todon$, the maps $$p_{B_n,B_{n+1}},g_{B_n,B_{n+1}}:|\mathcal{N}(B_{n+1})|\longrightarrow|\mathcal{N}(B_{n+1})|,$$ are homotopic. If $x\in|\mathcal{N}(B_{n+1})|$ then $x$ is in contained in a unique simplex $$\sigma=\langle B_{n+1}(a_0), B_{n+1}(a_1),\ldots,B_{n+1}(a_s)\rangle$$ with $$B_{n+1}(a_0)\cap B_{n+1}(a_1)\cap\ldots\cap B_{n+1}(a_s)\neq\emptyset.$$ Let us write $$p_{B_n,B_{n+1}}(\sigma)=\langle B_{n}(b_0), B_{n}(b_1),\ldots,B_{n}(b_s)\rangle$$ where, for every $i=0,\ldots s$, $b_i\subset q_{A_n}(a_i)$. Then $B_{n+1}(a_i)\subset B_n(b_i)$, and $$g_{B_n,B_{n+1}}(\sigma)=\langle B_{n}(c_0), B_{n}(c_1),\ldots,B_{n}(c_s)\rangle$$ where, for every $i=0,\ldots s$, $B_{n+1}(a_i)\subset B_n(c_i)$. It is clear now that $p_{B_n,B_{n+1}}(\sigma)\cup g_{B_n,B_{n+1}}(\sigma)$ is simplex of $\mathcal{N}(B_{n+1})$ because $$\emptyset\neq\bigcap_{i=0}^s B_{n+1}(a_i)\subset\bigcap_{i=0}^s B_{n}(b_i)\cap\bigcap_{i=0}^s B_{n}(c_i).$$ So the maps $p_{B_n,B_{n+1}},g_{B_n,B_{n+1}}$ are contiguous and hence homotopic. Then the identity is a morphism between the inverse sequences $\mathcal{\check{C}}_A(X)$ and $\mathcal{\check{C}}(X)$, so they are isomorphic and we are done. 
\end{proof}
\paragraph{The witness approximative sequence}
The witness complex is a simplicial complex constructed over a finite set of points with nice computational properties. Its simplices are sets of points which are close enough to a point that acts as a witness for them. They do not only depend on the 1-skeleton (as the Vietoris Rips complexes) and they do not produce so high dimensional simplexes as Vietoris Rips or \v{C}ech complexes. To see the definition and some properties of these complexes, see \cite{Ctopology}. Now, we define them in our context. Let us consider the Main Construction over the compact metric space $X$. For every $\todon$, consider the simplicial complex $\mathcal{W}_n$ whose vertex set are the points of the $\varepsilon_n$-approximation $A_n$ and the simplices are sets of points $\{a_0,a_1,\ldots,a_s\}\subset A_n$ such that every subset $\{a_{i_0},\ldots,a_{i_r}\}$ satisfies that there exists an $x\in X$, the witness, such that $$\sum_{j=0}^r\textrm{d}(x,a_{i_j})<(r+1)\varepsilon_n.$$ It is clear from the definition that this is indeed a simplicial complex. Now, we want to define maps between the witness complexes asociated to different approximations in order to define a sequence of polyhedra based on witness complexes. The idea here is that these maps are defined in a way that they \comillas{preserve} the witness for each set of points.  Let us define the map \mapeo{\omega_{n,n+1}}{\mathcal{W}_{n+1}}{\mathcal{W}_n}{a}{b\in q_{A_n}(a),} which is a simplicial map. Let us suppose that the simplex $\sigma=\langle a_0,a_1,\ldots, a_s\rangle$ is mapped to $\langle b_0,b_1,\ldots, b_s\rangle$. Consider the subset $\{b_{i_0},\ldots,b_{i_r}\}$. There exists a witness $x\in X$ for the corresponding subset $\{a_{i_0},\ldots,a_{i_r}\}$ of the simplex $\sigma$ and we claim that it is also a witness for its image. So, we estimate the sum
\begin{eqnarray*}\sum_{j=0}^r\textrm{d}(x,b_{i_j}) & \leqslant & \sum_{j=0}^r\left(\textrm{d}(x,a_{i_j})+\textrm{d}(a_{i_j},b_{i_j})\right)<(r+1)\varepsilon_{n+1}+\sum_{j=0}^r\gamma_n<\\ &<&(r+1)\frac{\varepsilon_n-\gamma_n}{2}+(r+1)\gamma_n=(r+1)\frac{\varepsilon_n+\gamma_n}{2}<(r+1)\varepsilon_n
\end{eqnarray*}  
and conclude that the map is simplicial. As in previous cases, the realization of this simplicial map is a well defined Hmap: if $b,b'\in q_{A_n}(a)$, then $\textrm{d}(b,a)+\textrm{d}(b',a)<2\varepsilon_n$ (here $a$ is acting as a witness to prove that $\langle b,b'\rangle$ is a simplex in $\mathcal{W}_n$), so the two possible definitions are contiguous maps and hence they are in the same homotopic class of maps.
We obtain then an inverse sequence of polyhedra called the \emph{witness approximative sequence}
$$\mathcal{W}(X)=\{|\mathcal{W}_n|,|\omega_{n,n+1}|\}.$$
As before, we will prove
\begin{prop}
The sequence $\mathcal{W}(X)$ is an HPol-expansion of $X$.
\end{prop}
\begin{proof}
We will see that it is isomorphic to $\mathcal{M}(X)$. For every $\todon$, the identitity map defined on the vertices of the witness complex
\mapeo{f_n}{\mathcal{W}_n}{\mathcal{R}_{2\varepsilon_n}(A_n)}{a}{a,} is a simplicial map. Indeed, if $\sigma=\langle a_0,a_1,\ldots,a_s\rangle$ then $f_n(\sigma)=\sigma$ is a simplex in $\mathcal{R}_{2\varepsilon_n}(A_n)$ since $\textrm{diam}\{a_0,a_1,\ldots,a_s\}<2\varepsilon_n$ because for every pair $a_i,a_j\in\sigma$ there exists a point $x\in X$ such that $$\textrm{d}(a_i,a_j)\leqslant\textrm{d}(a_i,x)+\textrm{d}(x,a_j)<2\varepsilon_n.$$ The realizations of these maps are a map between the sequences $\mathcal{W}(X)$ and $\mathcal{M}(X)$ since the diagram
$$\xymatrix@C=2cm@R=1.5cm{|\mathcal{W}_n|\ar[d]_{|f_n|} & |\mathcal{W}_{n+1}|\ar[d]^{|f_{n+1}|}\ar[l]_{|\omega_{n,n+1}|}\\
|\mathcal{R}_{2\epsilon_n}(A_n)| & |\mathcal{R}_{2\epsilon_{n+1}}(A_{n+1})|\ar[l]^{|p_{n,n+1}^*|}}$$ is commutative up to homotopy. Let $x$ be a point of $|\mathcal{W}_{n+1}|$, then it belongs to a unique simplex $\sigma=\langle a_0,a_1,\ldots,a_s\rangle$ of $\mathcal{W}_{n+1}$. Let us write the images of this simplex as 
$$f_n\cdot\omega_{n,n+1}(\sigma)=\langle b_0,b_1,\ldots,b_s\rangle\enspace\text{and}\enspace p_{n,n+1}^*\cdot f_{n+1}(\sigma)=\langle c_0,c_1,\ldots,c_s\rangle,$$
where $b_i,c_i\subset q_{A_n}(a_i)$ for $i=0,1,\ldots,s$. Then $\langle b_0,b_1,\ldots,b_s, c_0,c_1,\ldots,c_s\rangle$ is a simplex of $\mathcal{R}_{2\varepsilon_n}(A_n)$ because for every pair of vertexes $b_i,c_j$ we have that
\begin{eqnarray*}\textrm{d}(b_i,c_j) & \leqslant & \textrm{d}(b_i,a_i)+\textrm{d}(a_i,a_j)+\textrm{d}(a_j,b_j)<\\ & < & 2\gamma_n+2\gamma_{n+1}<\varepsilon_n+\gamma_n<2\varepsilon_n.
\end{eqnarray*} 
So, the two compositions are contiguous, hence its realizations homotopic.
Now, in order to apply Morita's lemma, we define the following simplicial maps, for each $\todon$,\mapeo{g_n}{\mathcal{R}_{2\varepsilon_{n+1}}(A_{n+1})}{\mathcal{W}_n}{a}{b\in q_{A_n}(a),} whose realizations make the following diagram commutative, for every $\todon$:
$$\xymatrix@C=2cm@R=1.5cm{|\mathcal{W}_n|\ar[d]_{|f_n|} & |\mathcal{W}_{n+1}|\ar[d]^{|f_{n+1}|}\ar[l]_{|\omega_{n,n+1}|}\\
|\mathcal{R}_{2\epsilon_n}(A_n)| & |\mathcal{R}_{2\epsilon_{n+1}}(A_{n+1})|\ar[l]^{|p_{n,n+1}^*|}\ar[lu]_{|g_n|}}$$ 
We have to prove several facts. First of all, the map just defined is simplicial. Let $\sigma=\langle a_0,a_1,\ldots,a_s\rangle$ be a simplex of $\mathcal{R}_{2\varepsilon_{n+1}}(A_{n+1})$, and write $g_n(\sigma)=\langle b_0,b_1,\ldots,b_s\rangle$. Consider any subset $\{b_{i_0},\ldots,b_{i_r}\}$. In this case, we will not use the witness for the corresponding subset $\{ a_0,a_1,\ldots,a_s\}$. We will just use $x=a_0$ as witness. So, we evaluate the sum 
$$\sum_{j=0}^r\textrm{d}(b_{i_j},a_0)\leqslant\sum_{j=0}^r\left(\textrm{d}(a_0,a_{i_j})+\textrm{d}(a_{i_j},b_{i_j})\right)<\sum_{j=0}^r\left(2\varepsilon_{n+1}+\gamma_n\right)<(r+1)\varepsilon_n$$ 
and conclude that $g_n(\sigma)$ is a simplex of the witness complex. Again, the realization of this simplicial map is a well defined up to homotopy map, because for two different $b,b'\in q_{A_n}(a)$, $\langle b,b'\rangle$ is a simplex of $\mathcal{W}_n$. It only remains to prove that, in fact, the two triangular diagrams abobe commute up to homotopy. For the upper-right one, consider the simplex $\sigma=\langle a_0,a_1,\ldots,a_s\rangle$ of $\mathcal{W}_{n+1}$ and write for its images $\omega_{n,n+1}(\sigma)=\langle b_0,b_1,\ldots,b_s\rangle$ and $g_n\cdot f_{n+1}(\sigma)=\langle b'_0,b'_1,\ldots,b'_s\rangle.$ The union of the images, $\langle b_0,b_1,\ldots,b_s,b'_0,b'_1,\ldots,b'_s\rangle,$ is a simplex of $\mathcal{W}_n$.  The subset $\{b_{i_0},\ldots,b_{i_{r_1}},b'_{j_0},\ldots,b'_{j_{r_2}}\}$ has a corresponding one $\{a_{i_0},\ldots,a_{i_{r_1}},a_{j_0},\ldots,a_{j_{r_2}}\}$ with $x\in X$ as witness, so
$$\sum_{k=0}^{r_1}\textrm{d}(a_{i_k},x)+\sum_{k=0}^{r_2}\textrm{d}(a_{j_k},x)<(r_1+r_2+2)\varepsilon_n.$$
Then, \begin{eqnarray*}
\sum_{k=0}^{r_1}\textrm{d}(b_{i_k},x)+\sum_{k=0}^{r_1}\textrm{d}(b_{j_k},x)&\leqslant &\sum_{j=0}^{r_1}\textrm{d}(b_{i_k},a_{i_k})+\textrm{d}(a_{i_k},x)+\sum_{j=0}^{r_2}\textrm{d}(b_{j_k},a_{j_k})+\textrm{d}(a_{j_k},x)<\\&<&(r_1+r_2+2)\varepsilon_{n+1}+(r_1+r_2+2)\gamma_n<(r_1+r_2+2)\varepsilon_n.
\end{eqnarray*}
so $\omega_{n,n+1}$ and $g_n\cdot f_{n+1}$ are contiguous maps and their realizations homotopic. To prove the lower left commutativity, we consider a simplex $\sigma=\langle a_0,a_1,\ldots,a_s\rangle$ and observe that the images $f_n\cdot g_n(\sigma)=\{b_0,\ldots,b_s\}$ and $p_{n,n+1}^*(\sigma)=\{b'_0,\ldots,b'_s\}$ are in the same simplex (their union). This is so, because for any $b_i,b'_j$ in that union, we have,
$$\textrm{d}(b_i,b'_j)\leqslant\textrm{d}(b_i,a_i)+\textrm{d}(a_i,a_j)+\textrm(a_j,b'_j)<2\gamma_n+2\varepsilon_{n+1}<2\varepsilon_n,$$
and that means that the diameter of the union makes it a simplex and then the maps $f_n\cdot g_n$ and $p_{n,n+1}^*$ are contiguous.
\end{proof}
From this proof, it is readily seen that every simplex of the witness complex is indeed a simplex of the Vietoris Rips complex. But the converse is not true, so the witness complex always will have less simplexes (of less or equal dimension) than the Vietoris Rips. Moreover, with the idea of approximation of compact metric spaces in mind, it makes sense to consider a set of points a simplex only if there is a point close enough to all of them. This make its use better for simplicity and computational purposes.

\paragraph{The Dowker approximative sequence}
Let us recall the simplicial complexes defined by Dowker in \cite{Dhomology}, for a given relation on two sets. Given two sets $X$ and $Y$ and a relation between the two sets $R$, i.e., a subset of the cartesian product $R\subset X\times Y$, we define two simplicial complexes $K^X$ and $L^Y$. On the one hand, a finite subset $\sigma$ of elements of $X$ is a simplex of $K^X$ if there exists an element $y\in Y$ related with every element  $x\in\sigma$. On the other hand, a finite subset $\tau$ of elements of $Y$ is a simplex of $L^Y$ if there exists an element $x\in X$ related with every element of $y\in\sigma$. Note that there is a kind of duality in these definitions. It is readily seen that $K^X$ and $K^Y$ are simplicial complexes. In that paper it is shown that this two complexes have the same homology. Moreover it is proven the following
\begin{teo}[Dowker \cite{Dhomology}]\label{teo:dowker}
The realizations of the simplicial complexes $|K^X|$ and $|L^Y|$ have the same homotopy type.
\end{teo}
In \cite{Dhomology} this is used to prove that the \v{C}ech and Vietoris homology for general topological spaces are isomorphic. Moreover, in Shape Theory, it is used to show that the standard \v{C}ech and Vietoris systems for any topological space are isomorphic. The power of Dowker Theorem lies in the generality of its formulation. We only need two sets and a relation and we obtain two homotopical simplicial complexes.
We can reformulate it in the context of Alexandrov spaces as follows. Consider an Alexandrov $T_0$ space given by the poset $(X,\leqslant)$. Let us consider the relation $R\subset X\times X$ given by $x R y\Leftrightarrow x\leqslant y$. Then, we have the following simplicial complexes:
\begin{eqnarray*}
\sigma=\langle x_0,\ldots,x_r\rangle\in K^X & \Longleftrightarrow & \exists y\in X: x_0,\ldots,x_r\leqslant y\\
\tau=\langle y_0,\ldots,y_s\rangle\in L^X & \Longleftrightarrow & \exists x\in X: x\leqslant y_0,\ldots,y_s
\end{eqnarray*}
Recall that, for every Alexandrov $T_0$ space, we can construct the McCord complex, which, with the same notation, we can write as
$$\rho=\langle z_0,\ldots,z_p\rangle\in \mathcal{K}(X)\Longleftrightarrow z_0\leqslant\ldots\leqslant z_p.$$
\begin{obs}
As simplicial complexes, $\mathcal{K}(X)\subset K^X,L^X$.
\end{obs}
Now, we adapt this to our special context. Let $X$ be a compact metric space and suppose the Main Construction done. Consider, for every $\todon$, the finite $T_0$ spaces $X_n=U_{2\varepsilon_n}(A_n)$ and the relation defined above with the order given by the upper semifinite topology, i.e., $$C R D\Longleftrightarrow C\subset D.$$ The simplicial complexes $K^{X_n},L^{X_n}$ are 
\begin{eqnarray*}
\sigma=\langle C_0,\ldots,C_r\rangle\in K^{X_n} & \Longleftrightarrow & \exists D\in X_n: C_0,\ldots,C_r\subset D,\\
\tau=\langle D_0,\ldots,D_s\rangle\in L^{X_n} & \Longleftrightarrow & \exists C\in X_n: C\subset D_0,\ldots,D_s.
\end{eqnarray*}
With the same notation, the McCord complex $\mathcal{K}(X)$ is defined by
$$\rho=\langle F_0,\ldots,F_p\rangle\in \mathcal{K}(X)\Longleftrightarrow F_0\subset F_1\subset\ldots\subset F_p.$$
We would like to define inverse sequences based on this simplicial complexes. To do so, we observe that the continuous maps $p_{n,n+1}:X_{n+1}\rightarrow X_n$ can be used to define simplicial maps:
\begin{minipage}{8cm}
 \mapeo{p^K_{n,n+1}}{K^{X_{n+1}}}{K^{X_n}}{C}{p_{n,n+1}(C),}
\end{minipage}
\begin{minipage}{7cm}
\mapeo{p^L_{n,n+1}}{L^{X_{n+1}}}{L^X_n}{C}{p_{n,n+1}(C).}
\end{minipage}
\\ \\
These maps are simplicial since 
\begin{eqnarray*}
C_0,\ldots,C_r\subset D & \Longrightarrow & p_{n,n+1}(C_0),\ldots,p_{n,n+1}(C_r)\subset p_{n,n+1}(D),\\
D\subset C_0,\ldots,C_s & \Longrightarrow & p_{n,n+1}(D)\subset p_{n,n+1}(C_0),\ldots,p_{n,n+1}(C_s).
\end{eqnarray*}
This allows us to define the inverse sequences of polyhedra:
\[
\mathcal{D}_u(X)=\{|K^{X_n}|,|p^K_{n,n+1}|\},
\enspace\enspace
\mathcal{D}_l(X)=\{|L^{X_n}|,|p^L_{n,n+1}|\},
\]
called respectively the \emph{upper and lower Dowker approximative sequences}.
\begin{prop}
The upper and lower Dowker approximative sequences are HPol-expansions of $X$.
\end{prop}
\begin{proof}
We will just prove that $K^X$ is an HPol-expansion. The proof for $L^X$ is completely dual (in the sense that the proof is exactly the same but using de dual property that defines the simplices for the complexes in this case). In order to prove this, we will see that $K^X$ is isomorphic, as HPol-sequence, to the approximative Alexandroff-McCord sequence. First of all, we see that, for every $\todon$, the inclusion map $i_n$ of $\mathcal{K}(X_n)$ in $K^{X_n}$ give us a morphism between the corresponding sequences. Indeed, for every $\todon$, it is easy to see that the following diagram commutes (we just need them to commute up to homotopy, but actually both compositions are exactly the same map)
$$\xymatrix@C=2cm@R=1.5cm{|\mathcal{K}(X_n)|\ar[d]_{|i_n|} & |\mathcal{K}(X_{n+1})|\ar[l]_{|p_{n,n+1}|}\ar[d]^{|i_{n+1}|}\\
|K^{X_n}| & |K^{X_{n+1}}|\ar[l]^{|p_{n,n+1}^K|}}$$In order to apply Morita's lemma, we would like to define a (diagonal) map from $K^{X_{n+1}}$ to $\mathcal{K}(X_n)$ but it seems there is no (appropiate) simplicial map between these spaces. Alternatively, we can define a simplicial map from the barycentric subdivision, \mapeo{g_n}{(K^{X_{n+1}})'}  {\mathcal{K}(X_{n+1})}{\{C_0,\ldots,C_s\}}{\bigcup_{i=0}^s p_{n,n+1}(C_i),}where there exists $C\in X_n$, such that, $\{C_0,\ldots,C_s\}\subset C$, so we have that $$\bigcup_{i=0}^s p_{n,n+1}(C_i)\subset p_{n,n+1}(C),$$ hence the map is well defined. To see that it is simplicial, let us write some notation. Let $C^j=\{C_0^j,C_1^j,\ldots,C_j^{s_j}\}$ be a set of points $C_i^j\in X_{n+1}$, we can write a simplex of $(K^{X_{n+1}})'$ as $$\langle C^0,C^0\cup C^1,\ldots,C^0\cup C^1\cup C^r\rangle.$$ The image of this simplex by $g_n$ is 
$$\left\langle\bigcup_{i=0}^{s_0}p_{n,n+1}(C_i^0),\bigcup_{j=0}^1\bigcup_{i=0}^{s_j} p_{n,n+1}(C_i^j),\ldots,\bigcup_{j=0}^{r}\bigcup_{i=0}^{s_j} p_{n,n+1}(C_i^0)\right\rangle,$$which is a simplex of $\mathcal{K}(X_n)$ since, for every $k=0,\ldots,r-1$, $$\bigcup_{j=0}^k\bigcup_{i=0}^{s_j} p_{n,n+1}(C_i^j)\subset\bigcup_{j=0}^{k+1}\bigcup_{i=0}^{s_j} p_{n,n+1}(C_i^j).$$ 
Now we are going to prove that the realization of this map satisfies the Morita's lemma making the following diagram commutative:
$$\xymatrix@C=2cm@R=1.5cm{|\mathcal{K}(X_n)|\ar[d]_{|i_n|} & |\mathcal{K}(X_{n+1})|\ar[l]_{|p_{n,n+1}|}\ar[d]^{|i_{n+1}|}\\
|K^{X_n}| & |K^{X_{n+1}}|\ar[l]^{|p_{n,n+1}^K|}\ar[ul]_{|g_n|}}$$
In order to prove this we need to use the barycentric subdivision of our simplicial complexes and the map $\rho:K'\rightarrow K$, defined above and shown to be isomorphic to the identify, for every simplicial complex $K$. Moreover, every simplicial map $f:K\rightarrow L$ induces a simplicial map between its barycentric subdivisions $f':K'\rightarrow L'$ (see \cite{Mfinitecomplexes} for details). Using this, we define the map $$(i_{n,n+1})':\mathcal{K}'(X_{n+1})\longrightarrow (K^{X_{n+1}})'$$ which is the simplicial map induced in the barycentric subdivisions by $i_{n,n+1}$ and 
$$(p_{n,n+1})':\mathcal{K}'(X_{n+1})\longrightarrow\mathcal{K}(X_n),$$ that is the composition $p_{n,n+1}\cdot\rho$. Finally, $$(p_{n,n+1}^K)':(K^{X_{n+1}})'\longrightarrow K^{X_n}$$ is the composition $p_{n,n+1}^K\cdot\rho$. Now, we just need to show that the following diagrams of simplicial maps are contiguous:
$$\xymatrix@C=2cm@R=1.5cm{\mathcal{K}(X_n)\ar[d]_{i_n} & \mathcal{K}'(X_{n+1})\ar[l]_{(p_{n,n+1})'}\ar[d]^{(i_{n+1})'}\\
K^{X_n} & (K^{X_{n+1}})'\ar[l]^{(p_{n,n+1}^K)'}\ar[ul]_{g_n}}$$
The upper right diagram is not only contiguous but commutative. With the notation on simplexes above, if $\sigma=\langle C^0,C^0\cup C^1,\ldots,C^0\cup\ldots\cup C^r\rangle\in\mathcal{K}(X_{n+1})$ with $C_i^j\subset C_{i+1}^j$ for every $j=0,\ldots,r$ and $i=0,\ldots,s_j-1$, is a simplex of $\mathcal{K}'(X_{n+1})$, then $$i_n\cdot p_{n,n+1}'(\sigma)=\langle p_{n,n+1}(C_{s_0}^0),p_{n,n+1}(C_{s_1}^1),\ldots,p_{n,n+1}(C_{s_r}^r\rangle=({p_{n,n+1}^k})'(\sigma).$$
The lower left diagram is contiguous because if $\tau=\langle C^0,C^0\cup C^1,\ldots,C^0\cup\ldots\cup C^r\rangle$ is a simplex of $(K^{X_{n+1}})'$, then $$i_n\cdot g_n(\tau)=\langle p_{n,n+1}(C^0_{s_0}),p_{n,n+1}(C^0_{s_0}\cup C^1_{s_1}),\ldots,p_{n,n+1}(C_{s_0}^0\cup\ldots\cup C_{s_r}^r)\rangle,$$ and $$(p_{n,n+1}^K)'(\tau)=\langle p_{n,n+1}(C^0_{s_0}),p_{n,n+1}(C^1_{s_1}),\ldots,p_{n,n+1}(C_{s_r}^r)\rangle.$$ The union of both images is a simplex of $K^{X_n}$ since every vertex is contained in $p_{n,n+1}(C_{s_0}^0\cup\ldots\cup C_{s_r}^r)$. Thus the maps $i_n\cdot g_n$ and $(p_{n,n+1}^K)'$ are contiguous and hence its realizations homotopic, and we are done.
\end{proof}

All these inverse approximative sequences share the property that they are defined in terms of a sequence of adjusted finite approximations obtained by the Main Construction. Given a compact metric space $(X,\text{d})$, we will say that an inverse sequence of polyhedra $\{K_n,p_{n,n+1}\}$ is an \emph{polyhedral approximative sequence} (shortly, \textsc{pas}) of $X$, if it is any of the inverse sequences defined on this section. Now, we formalize the idea that these inverse sequences capture the shape type of the space.
\begin{teo}[General Principle]\label{teo:GP}
Every \textsc{pas} of a compact metric space $X$ is an HPol-expansion of $X$.
\end{teo}

\begin{cor}
Let $X$ be a compact metric space and $K=\{K_n,p_{n,n+1}\}$ a polyhedral approximative sequence of $X$. Consider the induced inverse sequence of groups $F(K)=\{F(K_n),F(p_{n,n+1})\}$, where $F$ is the singular $p$-th homology or homotopy functor (with some coefficients). Then, the inverse limit of $F(K)$ is the $p$-th \v{C}ech homology or shape group of $X$, respectively.
\end{cor}
\section{Inverse Persistent homology and its errors}\label{sec:persistenterrors}

Given a \textsc{pas}, we want to use it to compute persistence modules that can be used to detect topological festures orf the original space. The inverse persistence procedure introduced in \cite{MMreconstruction} is designed to achieve this using inverse sequences of polyhedra, relating persistent topology and shape. In this section, we define some concepts in this framework for our \textsc{pas} and study its behaviour with some examples. This procedures can be applied to pointclouds in the same way. Concretely, we will obtain homology persistence modules. The idea is to infer some information about the \v{C}ech homology of some compact metric space using finite cuts of the inverse sequence that defines it (in terms of inverse limit), in the same way that Taylor polynomials approximate non linear functions up to some error. Here, polyhedra will play the role of polynomials that approximate any compact metric space.

Let us consider, for a compact metric space $X$, a \textsc{pas} (or, in general, any HPol-expansion) $\{K_n,p_{n,n+1}\}$. Consider, for each polyhedron $K_i$, the $p$-th homology group $H_n:=H_p(K_n)$ and the induced maps $q_{n,n+1}:=(p_{n,n+1})_*$ and $q_{n,m}$ for the corresponding compositions for $n<m$. We do not specify the coefficients since we only need to use that the homology groups are abelian. In the concrete examples of this section, we use $\mathbb{Z}$.  We obtain therefore an inverse sequence of finitely generated abelian groups $\{H_n,q_{n,n+1}\}$
whose inverse limit is the $p$-th \v{C}ech homology group of $X$, written $H=\check{H}_p(X)$, with projection $q_n:=(f_n)_*$ for every $\todon$.

In these conditions, fix $\todon$ and define, for every $n<m$, the \emph{$(n,m)-th$ group of inverse persistent homology} as $$H_{n,m}:=\textrm{im}(q_{n,m})=q_{n,m}(H_m).$$ The inclusion map $i$ allows us to obtain an inverse sequence of persistent homology groups $\{H_{n,n+1},i\}$. As an inverse sequence defined by inclusion maps, the inverse limit is the intersection of all of them, $$\bigcap_{m=n+1}^\infty H_{n,m}=q_n(H).$$ The persistent group $H_{n,m}$ is a normal subgroup of $H_n$ for every $n<m$. So, we can take the quotient of groups $E_{n,m}:=\frac{H_n}{H_{n,m}}$, that we will call the \emph{$(n,m)$-th inverse persistent error}. The idea of this group is that it measures the validity of $H_n$ seen from the $H_m$ perspective. Moreover, as $H_{n,m+1}$ is also a normal subgroup of $H_{n,m}$, we therefore obtain a natural homomorphism between the quotients $$\begin{array}{llcl}g_{m,m+1}:&E_{n,m+1}&
\longrightarrow &E_{n,m}\nonumber\\&h+H_{n,m+1}&\longmapsto & h+H_{n,m}.\nonumber\end{array}$$(Or, using a different notation, $[h]_{H_{n,m+1}}
\mapsto[h]_{H_{n,m}}$). By composition, we obtain an inverse sequence of errors $$E_{n,n+1}
\xleftarrow{g_{n+1,n+2}}E_{n,n+2}\xleftarrow{g_{n+2,n+3}}\ldots\xleftarrow{g_{n+m-1,n+m}} E_{n+m,n+m+1}\xleftarrow{g_{n+m+1,n+m+2}}
\ldots,$$ with an inverse limit, denoted by $E_n^i$ and called the \emph{inductive $n$-error}. In an ideal of understanding this error in the infinity we 
define the \emph{$n$-th real error}, $E_n:=\frac{H_n}{q_n(H)}$. In general, these two errors in the limit we have just defined, $E_n^i$ and $E_n$,  are different, but we can do some comparations. 
\begin{prop}There is an injective homomorphism of groups $\varphi:E_n\rightarrow E_n^i$.
\end{prop}
\begin{proof}
We define the map with $$h+q_n(H)\longmapsto (h+H_{n,m},h+H_{n,m+1},\ldots).$$ It is well defined, because if $h-h^\prime\in q_n(H)=\bigcap H_{n,m}$ then $h-h^\prime$ represents the null class in every group $E_{n,m}$.
It is injective because if $$(h_1+H_{n,m}, h_1+H_{n,m+1},\ldots)=(h_2+H_{n,m}, h_2+H_{n,m+1},\ldots)$$ then $h_1-h_2\in H_{n,m},\enspace h_1-h_2\in H_{n,m+1},\ldots$, so
$h_1-h_2\in\bigcap H_{n,m}$ and $h_1+q_n(H)=h_2+q_n(H)\qedhere$
\end{proof}
Moreover, for some kind of spaces, this two errors are actually the same group.
\begin{prop}
Consider an HPol-expansion $\{K_n,p_{n,n+1}\}$ of a compact metric movable space $X$. Then we have $E_n^i=E_n$ for every $n\in\mathbb{N}$.
\end{prop}
\begin{proof}
If $X$ is movable, then, the inverse sequence that defines it as inverse limit is also movable, so it is its induced homology sequence. So, the last has the Mittag-Leffer property, that is: for every $n\in\mathbb{N}$, there exists $m\geqslant n$ such that, for every $r\geqslant m$, we have that $q_{n,r}(H_r)=q_{n,m}(H_m)$, i.e., expressed with persistent homology groups, $H_{n,r}=H_{n,m}$. So, from $m$, all the persistent groups are the same, so $q_n(H)=\bigcap H_{n,m}=H_{n,m}$. Therefore, the inverse sequence of errors from $m$ is constant and equal to $E_{n,m}=E_n=E_n^i$.
\end{proof}
Sometimes we can define a third kind of error. If there exists a homomorphism of groups $f:H_n\rightarrow H_n$
such that $f(H_{n,m})\subset H_{n,m+1}$, we can define the map $$\begin{array}{llcl}l_{m,m+1}:&E_{n,m}&
\longrightarrow &E_{n,m+1}\nonumber\\&h+H_{n,m}&\longmapsto & f(h)+H_{n,m+1}.\nonumber\end{array}$$ This map is well defined because if $h-h^\prime\in H_{n,m}$, we get $f(h)-f(h^\prime)=f(h-h^\prime)\in H_{n,m+1}$. By composition, we can form the direct sequence $$E_{n,n+1}
\xrightarrow{l_{n+1,n+2}}E_{n,n+2}\xrightarrow{l_{n+2,n+3}}\ldots\xrightarrow{l_{n+m-1,n+m}} E_{n+m,n+m+1}\xrightarrow{l_{n+m+1,n+m+2}}
\ldots,$$ and then, a direct limit, denoted by $E_n^d$ and called the \emph{direct $n$-error}. But not every inverse sequence satisfies this property. 
\begin{obs}For \textsc{pas} with constant bonding map $\{K_n,p\}$, there exists the limit $E_n^d$. In this particular case, the map $p^*:H_n\rightarrow H_n$ induced in the homology groups plays the role of $f$, because
$$p^*(H_{n,m})=p^*p_{n,m}^*(H_m)=p^*p^*_{n+1,m+1}(H_{m+1})=p^*_{n,m+1}(H_{m+1})=H_{n,m+1}.$$
\end{obs}
\begin{ej}The dyadic solenoid. This space can be defined as the inverse limit of the inverse sequence $\{\mathbb{S}^1,2\}$, considering $\mathbb{S}^1$ as the complex unit circle and the map \comillas{2} means the exponential map $z\mapsto z^2$. The induced homology inverse sequence of order 1 is $\{\mathbb{S}^1,2\}$, with inverse limit $\{0\}$, meaning that the \v{C}ech homology of the dyadic solenoid is trivial. If we consider here our errors, we obtain that for every $n<m$, the persistent homology groups are $H_{n,m}\simeq(2^{m-n})\mathbb{Z}$ (integer multiple of $2^{m-n}$). If we consider the quotients to obtain the $(n,m)$-errors, we obtain the finite groups $E_{n,m}\simeq\frac{\mathbb{Z}}{(2^{m-n})\mathbb{Z}}\simeq\mathbb{Z}_{2^{m-n}}$. The natural map between these errors sends each element to its class in the image group. Hence, the maps are
$$\begin{array}{lrcl}g_{m,m+1}:&\mathbb{Z}_{2^{m+1-n}}&
\longrightarrow &\mathbb{Z}_{2^{m+-n}}\nonumber\\&t&\longmapsto & t(\text{mod} 2).\nonumber\end{array}$$ With these maps, we can define the inverse sequence
$$\mathbb{Z}_2\longleftarrow\mathbb{Z}_4\longleftarrow\mathbb{Z}_8\longleftarrow\ldots\longleftarrow\mathbb{Z}_{2^n}\longleftarrow\ldots$$
which inverse limit $E_n^i\simeq\mathcal{D}_2$ is the dyadic integers group. In this case, we can also obtain the direct sequence construction. Here, the maps will send each element to the class of this element multiplied by two. That is, the maps are
$$\begin{array}{lrcl}l_{m,m+1}:&\mathbb{Z}_{2^{m+-n}}&
\longrightarrow &\mathbb{Z}_{2^{m+1-n}}\nonumber\\&t&\longmapsto & 2t.\nonumber\end{array}$$
So, the direct sequence $$\mathbb{Z}_2\longrightarrow\mathbb{Z}_4\longrightarrow\mathbb{Z}_8\longrightarrow\ldots\longrightarrow\mathbb{Z}_n
\longrightarrow\ldots$$ has direct limit the Prüfer 2-group $E_n^d\simeq\mathbb{Z}(2^{\infty})$ (the set of roots of the unity of some power of two). It turns out that the Prüfer group (with the discrete topology) is the Pontryagin dual of the compact group of the dyadic integers. 
\end{ej}
\begin{preg}
For what class of spaces can we obtain this duality of the errors and what is the meaning?
\end{preg}
\section{The Computational Hawaiian Earring}\label{sec:computationalhawaiianearring}
We now perform the Main Construction for a movable space called \emph{the Hawaiian Earring}, which is the subspace of $\mathbb{R}^2$,
$$\bigcup_{\todon\cup\{0\}}\bigcirc\left(\left(\frac{1}{2^n},\frac{1}{2}\right),\frac{1}{2^n}\right),$$
where $\bigcirc(a,\varepsilon)$ stands for the 1-sphere of center $a$ and radius $\varepsilon$ in $\mathbb{R}^2$. For computational reasons we are going to use not an homeomorphic copy of the Hawaiian Earring but an homotopic (and, hence, with the same shape) one. In this case, we consider the space (see Figure \ref{hawaiidef})
$$\mathcal{H}=\bigcup_{\todon\cup\{0\}}\square\left(\frac{1}{2^n},\frac{1}{2^n}\right),$$where $\square(a,b)$ stands for the square with diagonal the segment from $(0,0)$ to $(a,b)$ in $\mathbb{R}^2$.
\begin{figure}[h!]
\begin{minipage}[c]{0.45\textwidth}
\begin{center}
\begin{tikzpicture}[scale=3,domain=0:1]
\draw (1,0) circle (1cm);
\draw (0.5,0) circle (0.5cm);
\draw (0.25,0) circle (0.25cm);
\draw (1/8,0) circle (0.125cm);
\draw (1/16,0) circle (0.0625cm);
\draw (1/32,0) circle (0.03125cm);
\draw (1/64,0) circle (0.015625cm);
\end{tikzpicture}
\end{center}
\end{minipage}
\   \ \hfill
\begin{minipage}[c]{0.45\textwidth}
\begin{center}
\begin{tikzpicture}[scale=5,domain=0:1]
\draw (0,0)--(1,0)--(1,1)--(0,1)--cycle;
\draw (0,0.5)--(0.5,0.5)--(0.5,0);
\draw (0,0.25)--(0.25,0.25)--(0.25,0);
\draw (0,0.125)--(0.125,0.125)--(0.125,0);
\draw (0,1/16)--(1/16,1/16)--(1/16,0);
\draw (0,1/32)--(1/32,1/32)--(1/32,0);
\draw (0,1/64)--(1/64,1/64)--(1/64,0);
\end{tikzpicture}
\end{center}
\end{minipage}
\caption{The Hawaiian Earring and the Computational Hawaiian Earring.}
\label{hawaiidef}
\end{figure}
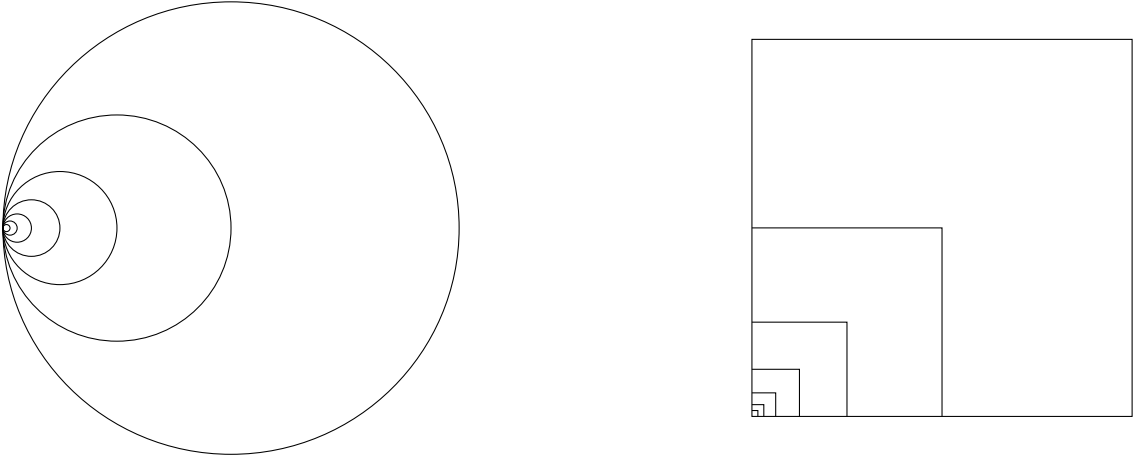

In order to perform the Main Construction, we consider points for each  approximation as the intersection of a grid of the corresponding side with the space, and add points where necessary. We pick $\varepsilon_1=2\sqrt{2}>\text{diam}(\mathcal{H})$ and then $A_1=\{(0,0)\}$. Obviously $\gamma_1=\sqrt{2}$ so we select $\epsilon_2=\frac{\sqrt{2}}{2^3}<\frac{\sqrt{2}}{2}$. We intersect a grid  $G_2$ of side $\frac{1}{2^2}$ with $\mathcal{H}$ (see Figure \ref{HEpol}) and adding the point $\left(\frac{1}{2^3},\frac{1}{2^3}\right)$, we obtain $A_2$, an $\epsilon_2$-approximation of $\mathcal{H}$. From the picture, $\gamma_2=\frac{1}{2^3}$ and then we can take $\epsilon_3=\frac{\sqrt{2}}{2^6}$. Intersecting a grid $G_3$ of side $\frac{1}{2^5}$ with $\mathcal{H}$ and adding the point $\left(\frac{1}{2^6},\frac{1}{2^6}\right)$, we obtain (see Figure \ref{HEpol}) an $\epsilon_3$-approximation $A_3$ of $\mathcal{H}$. 

By induction, we obtain that, for every $n>1$, the finite approximations are defined by
$$\epsilon_n=\frac{\sqrt{2}}{2^{3n-3}},\enspace, A_n=\left(G_n\cap\mathcal{H}\right)\cup\conjunto{\punto{\frac{1}{2^{3n-3}}}{\frac{1}{2^{3n-3}}}},\enspace\gamma_n=\frac{1}{2^{3n-3}}.$$

We therefore obtain the sequence of finite spaces from these approximations. Now we construct the Alexandroff-McCord \textsc{pas}. In the first step we have only a point as finite space, so the associated polyhedron is just a vertex. For the second and third step, we have depicted in Figure \ref{HEpol} the realization of the simplicial complexes $\mathcal{R}_{2\epsilon_2}(A_2)$ and $\mathcal{R}_{2\epsilon_3}(A_3)$.
\begin{figure}[h!]
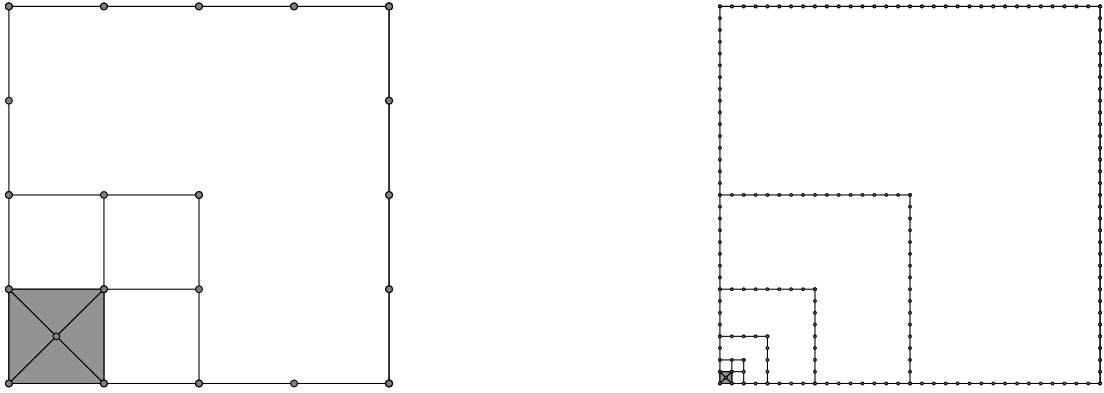

\begin{minipage}[c]{0.45\textwidth}
\begin{center}
\hawaiiPol{3}
\end{center}
\end{minipage}
\   \ \hfill
\begin{minipage}[c]{0.45\textwidth}
\begin{center}
\hawaiiPol{6}
\end{center}
\end{minipage}
\caption{The realization of the simplicial complexes $\mathcal{R}_{2\epsilon_2}(A_2)$ and $\mathcal{R}_{2\epsilon_3}(A_3)$.}
\label{HEpol}
\end{figure}
Then, the corresponding associated simplicial complexes are exactly the barycentric subdivisons of these complexes, that is, $\mathcal{K}\left(U_{2\epsilon_2}(A_2)\right)=\mathcal{R}'_{2\epsilon_2}(A_2)$ and $\mathcal{K}\left(U_{2\epsilon_3}(A_3)\right)=\mathcal{R}'_{2\epsilon_3}(A_3)$, respectively. We see that they consist of ()from north east to south west) a series of, let us say, inverted L's (one in $\mathcal{K}\left(U_{2\epsilon_2}(A_2)\right)$ and four in $\mathcal{K}\left(U_{2\epsilon_3}(A_3)\right)$), one more inverted L formed by three squares and one more square, filled. It is easy to see that, in each step, we just insert three more inverted L's and reduce the size of the three squares forming the inverted L and of the filled one. So, we can infer that, in general, the associated McCord polyhedron $\mathcal{K}\left(U_{2\epsilon_n}(A_n)\right)$ has $3n-5$ inverted L's, three squares forming an inverted L and one filled square. As before, the maps can be known just overlapping the polyhedra to see where each vertex is sent. But this will be better understood studying the homological situation.

We study homology where it makes sense, so here we just focus in homology of dimension 1. So, first of all, we give names to the generators. In the simplicial complex associated to the $\epsilon_i$-approximation, $\mathcal{K}(U_{2\epsilon_i}(A_i))$, let us call $\rho^j_i$ to the homology generator representing the $j$-th inverted L, counting from north east to south west, and $\lambda_i$, $\mu_i$, $\nu_i$ to the three squares, counted clockwise. Now, in the first non trivial case, the homology is $H_1\left(\mathcal{K}(U_{2\epsilon_2}(A_2))\right)\simeq\mathbb{Z}^4$ with generators $\rho^1_2,\lambda_2,\mu_2,\nu_2$. The next one is $H_1\left(\mathcal{K}(U_{2\epsilon_3}(A_3))\right)\simeq\mathbb{Z}^7$ with generators $\rho^i_3,\enspace\textrm{with}\enspace i=1,\ldots,4,\enspace\lambda_3,\mu_3,\nu_3$. The induced map $$(p_{2,3})_*:H_1(\mathcal{K}(U_{2\epsilon_3}(A_3)))\longrightarrow H_1(\mathcal{K}(U_{2\epsilon_2}(A_2))),$$ sends
\begin{align*}
\rho^1_3&\longmapsto \rho^1_2\\
\rho^2_3&\longmapsto \lambda_2+\mu_2+\nu_2\\
\rho^3_3,\rho^4_3,\lambda_3,\mu_3,\nu_3&\longmapsto 0.
\end{align*}
Interpreting this, we see that, in the third approximation, three \comillas{new} cycles are created, one is sent to the sum of the three squares, and the other two are sent to zero. This is repeated along all the sequence. For any $\todon$, we have that the first homology group is $H_1\left(\mathcal{K}(U_{2\epsilon_n}(A_{n}))\right)\simeq\mathbb{Z}^{3n-2}$, with generators $\rho^i_n,\enspace\textrm{with}\enspace i=1,\ldots,3n-5,\enspace\lambda_n,\mu_n,\nu_n$. The map $$(p_{n,n+1})_*:H_1(\mathcal{K}(U_{2\epsilon_{n+1}}(A_{n+1})))\longrightarrow H_1(\mathcal{K}(U_{2\epsilon_n}(A_n))),$$ acts sending
\begin{align*}
\rho^1_{n+1}&\longmapsto \rho^1_n\\
&\vdots\\
\rho^{3n-5}_{n+1}&\longmapsto \rho^{3n-5}_{n}\\
\rho^{3n-4}_{n+1}&\longmapsto \lambda_n+\mu_n+\nu_n\\
\rho^{3n-3}_{n+1},\rho^{3n-2}_{n+1},\lambda_{n+1},\mu_{n+1},\nu_{n+1}&\longmapsto 0.
\end{align*}
It is easy to understand the beaviour of the sequence. In each step, 3 \comillas{new} cycles are created, and they have an unique preimage in every further step. So the inverse limit of the sequence is $$\lim_{\leftarrow}\left\lbrace H_1\left(\mathcal{K}(U_{2\epsilon_n}(A_{n}))\right),(p_{n,n+1})_*\right\rbrace\simeq\mathbb{Z}^{\infty}$$ and, by Theorem \ref{teo:mccordexpansion}, the \v{C}ech homology of the space $\mathcal{H}$ is, $\check{H}_1(\mathcal{H})\simeq\mathbb{Z}^{\infty}$, as we already knew.

\paragraph{Persistent errors in the computational Hawaiian Earring}
The Hawaiian earring is a movable not stable space. As a movable space, the inverse sequence $\left\lbrace\mathcal{K}_n,(p_{n,n+1})_*\right\rbrace$ has the \textsc{ml} property. An easy induction tells us that, for every $m>n$, $\textrm{im}\parentesis{p_{n,m}}_*\simeq\mathbb{Z}^{3n-4}$, so, the \textsc{ml} index for every $\todon$ is, again, $n+1$. Let us compute the persistent errors. The $(n,m)$-th inverse persistent homology group, for $1<n<m$, is $H_{n,m}\simeq\mathbb{Z}^{3n-4}$, so we can compute the $(n,m)$-th persistent error $$E_{n,m}=\frac{H_n}{H_{n,m}}\simeq\frac{\mathbb{Z}^{3n-2}}{\mathbb{Z}^{3n-4}}\simeq\mathbb{Z}^2.$$
This two copies of $\mathbb{Z}$ represent the fact that, for every $\todon$, all the generators of the group $H_1(\mathcal{K}_n)$, but two, have one (and only one) preimage in every $H_1(\mathcal{K}_n)$. Considering $$\conjunto{\rho_n^1,\ldots,\rho_n^{3n-5},\lambda_n+\mu_n+\nu_n,\mu_n,\nu_n}$$ as generators, the last two are the ones without preimage. The map induced in $E_{n,m}$ is clearly the identity map, $$g_{m,m+1}=\textrm{id}:\mathbb{Z}^{2}\longrightarrow\mathbb{Z}^{2}$$ and the inverse limit of the inverse sequence $\conjunto{E_{n,m},g_{n+m,n+m+1}}$ is $E_n^i\simeq\mathbb{Z}^{2}$. Since $\mathcal{H}$ is movable, this error is equal to the $n$-th real error, $$E_n=\frac{H_n}{(p_n)_*(\check{H}_1(\mathcal{H}))}\simeq\mathbb{Z}^{2}.$$
Hence, we conclude that our inverse sequence creates two unnecessary generators in each step to define the Hawaiian Earring, as we knoww from the classical definition of it by inverse limits.

\paragraph{Acknowledgements}
The author wish to thank the valuable help, comments and support from his thesis advisor M.A. Morón for the results obtained in this paper.
\addcontentsline{toc}{chapter}{References}
\bibliographystyle{siam}
\bibliography{bibPolexp}
\end{document}